\documentclass{amsart}
\usepackage[utf8]{inputenc}
\usepackage{amsmath,amssymb,amsthm, graphics, comment,bm}
\usepackage{mathtools}
\usepackage{amsthm}
\usepackage{diagmac2}
\usepackage{textcomp}
\usepackage{color}
\usepackage[alphabetic]{amsrefs}
\usepackage[all,cmtip,color,matrix,arrow]{xy}
\usepackage{yfonts}
\usepackage{tikz}
\usepackage{stmaryrd}
\usepackage{mathrsfs}
\usepackage{enumerate}
\usepackage{hyperref}
\usepackage[mathscr]{euscript}

\usepackage{tikz-cd}
\usepackage{tikz}

\calclayout

\newcommand{\bA}{\mathbb{A}}
\newcommand{\cA}{\mathcal{A}}
\newcommand{\Proj}{\operatorname{Proj}}
\newcommand{\bC}{\mathbb{C}}
\newcommand{\bF}{\mathbb{F}}

\newcommand{\bL}{\mathbb{L}}

\newcommand{\bP}{\mathbb{P}}
\newcommand{\bQ}{\mathbb{Q}}

\newcommand{\bZ}{\mathbb{Z}}

\newcommand{\sC}{\mathscr{C}}
\newcommand{\cJ}{\mathcal{J}}

\newcommand{\sD}{\mathscr{D}}

\newcommand{\sF}{\mathscr{F}}
\newcommand{\sG}{\mathscr{G}}
\newcommand{\sH}{\mathscr{H}}
\newcommand{\sI}{\mathscr{I}}
\newcommand{\sK}{\mathscr{K}}
\newcommand{\sL}{\mathscr{L}}

\newcommand{\sP}{\mathscr{P}}
\newcommand{\sS}{\mathscr{S}}

\newcommand{\sX}{\mathscr{X}}

\newcommand{\cQ}{\mathcal{Q}}

\newcommand{\cD}{\mathcal{D}}
\newcommand{\cE}{\mathcal{E}}
\newcommand{\cF}{\mathcal{F}}
\newcommand{\cG}{\mathcal{G}}
\newcommand{\cH}{\mathcal{H}}
\newcommand{\cI}{\mathcal{I}}
\newcommand{\cK}{\mathcal{K}}
\newcommand{\cL}{\mathcal{L}}
\newcommand{\cM}{\mathcal{M}}

\newcommand{\cO}{\mathcal{O}}

\newcommand{\cS}{\mathcal{S}}

\newcommand{\cW}{\mathcal{W}}
\newcommand{\cX}{\mathcal{X}}
\newcommand{\cY}{\mathcal{Y}}
\newcommand{\cZ}{\mathcal{Z}}

\newcommand{\spec}{\operatorname{Spec}}

\newcommand{\sY}{\mathscr{Y}}
\newcommand{\supp}{\operatorname{Supp}}
\newcommand{\Ker}{\operatorname{Ker}}

\newcommand{\Sym}{\operatorname{Sym}}

\newcommand{\Gr}{\operatorname{Gr}}

\newcommand{\PGL}{\operatorname{PGL}}

\newcommand{\Id}{\operatorname{Id}}

\newtheorem{theorem}{Theorem}[section]

\newtheorem{Lemma}[theorem]{Lemma}
\newtheorem*{Oss'}{Observation}
\newtheorem{Cor}[theorem]{Corollary}
\newtheorem{Prop}[theorem]{Proposition}

\theoremstyle{definition}

\newtheorem{Def}[theorem]{Definition}
\newtheorem{Notation}[theorem]{Notation}
\newtheorem{Remark}[theorem]{Remark}

\newtheorem{Oss}[theorem]{Observation}

\DeclareMathOperator{\Supp}{Supp}
\newcommand{\rar}{\rightarrow}	
\newcommand{\drar}{\dashrightarrow}
\DeclareMathOperator{\vol}{vol}

\def\O#1.{\mathcal {O}_{#1}}			
\def\pr #1.{\mathbb P^{#1}}				
\def\af #1.{\mathbb A^{#1}}			
\def\ses#1.#2.#3.{0\to #1\to #2\to #3 \to 0}	
\def\xrar#1.{\xrightarrow{#1}}			
\def\K#1.{K_{#1}}						
\def\bA#1.{\mathbf{A}_{#1}}			
\def\bM#1.{\mathbf{M}_{#1}}				
\def\bL#1.{\mathbf{L}_{#1}}				
\def\bB#1.{\mathbf{B}_{#1}}				
\def\bK#1.{\mathbf{K}_{#1}}			
\def\subs#1.{_{#1}}					
\def\sups#1.{^{#1}}

\renewcommand{\qq}{\mathbb{Q}}
\newcommand{\zz}{\mathbb{Z}}
\newcommand{\nn}{\mathbb{N}}
\newcommand{\rr}{\mathbb{R}}

\makeatletter
\@namedef{subjclassname@2020}{
  \textup{2020} Mathematics Subject Classification}
\makeatother


\begin{document}

\title[Moduli of $\qq$-Gorenstein pairs and applications]{Moduli of $\qq$-Gorenstein pairs and applications}

\author[S.~Filipazzi]{Stefano Filipazzi}
\address{
EPFL, SB MATH CAG, MA C3 625 (B\^{a}timent MA), Station 8, CH-1015 Lausanne, Switzerland
}
\email{stefano.filipazzi@epfl.ch}

\author[G.~Inchiostro]{Giovanni Inchiostro}
\address{
University of Washington, Department of Mathematics, Box 354350, Seattle, WA 98195 USA
}
\email{ginchios@uw.edu}

\begin{abstract}
We develop a framework to construct moduli spaces of $\qq$-Gorenstein pairs.
To do so, we fix certain invariants; these choices are encoded in the notion of \emph{$\bQ$-stable pair}.
We show that these choices give a proper moduli space with projective coarse moduli space and they prevent some pathologies of the moduli space of stable pairs when the coefficients are smaller than $\frac{1}{2}$.
Lastly, we apply this machinery to provide an alternative proof of the projectivity of the moduli space of stable pairs.
\end{abstract}
\thanks{
SF was partially supported by ERC starting grant \#804334.
}

\maketitle

\section{Introduction}
Since the seminal work of Mumford on the moduli space $\cM_g$ of smooth curves with $g \geq 2$ and its compactification $\overline{\cM}_g$, significant progress has been made in understanding its higher-dimensional analog, namely the moduli space of stable varieties.
On a first approximation, those are $\mathbb{Q}$-Gorenstein varieties with relatively mild singularities, and such that $K_X$ is ample.
A complete and satisfactory moduli theory for stable varieties of any dimension has been settled due to the work of several mathematicians (see \cites{KSB, Kol90, Ale94, Vie95, Kol13-mod, Kol13, HK04, AH11, kol08, HX13, kol_new, HMX }).
Furthermore, in this setup, it is well understood what families should be parametrized by the moduli functor.

A natural generalization of $\cM_g$ is given by the moduli space of $n$-pointed curves and its compactification, denoted by $\cM_{g,n}$ and $\overline{\cM}_{g,n}$, respectively.
A first attempt to generalize the notion of stable pointed curve is to consider mildly singular pairs $(X,D)$ (specifically, \emph{semi-log canonical} pairs) where $D$ is a reduced divisor, such that $K_X+D$ is ample.
This approach has been worked out Alexeev in dimension 2 \cites{Ale94, Ale96}, and combining the efforts of several mathematicians (see \cites{ Ale94, Kol13, kol08, HX13, KP17, kol_new, HMX, kol19s }), it has been generalized in higher dimensions.

While a stable variety is $\mathbb{Q}$-Gorenstein, a stable pair might not be, so if one is interested in $\mathbb{Q}$-Gorenstein pairs, the aforementioned formalism might not give the desired moduli space.
In fact, one cannot simply consider the $\qq$-Gorenstein locus of the moduli of stable pairs, as this could be not proper.
Therefore, the main goal of this paper is to construct a \emph{proper} moduli space parametrizing $\mathbb{Q}$-Gorenstein \emph{pairs} $(X,D)$.
We observe that, in \cite{Ale15}, Alexeev already considered certain moduli spaces of $\mathbb{Q}$-Gorenstein stable pairs, see also \cite{kol_new}*{\S~6.4}.
As in the classical case, we require our pairs to have semi-log canonical singularities, but we relax the ampleness condition on $K_X+D$,
requiring only that $K_X+(1-\epsilon)D$ is ample for $0< \epsilon \ll 1$.
We also impose a numerical condition on the intersection number $(K_X + tD)^{\dim(X)}$, analogous to the genus condition for curves, which we encode via a polynomial function $p(t)$.
This condition guarantees the boundedness of our moduli problem.
These choices are encoded in the notion of \emph{$\bQ$-stable pair}, a pair on which both $K_X$ and $D$ are $\bQ$-Cartier, (Defintion \ref{def_p-pair}) and lead to the following:
\begin{theorem} \label{thm intro 1}
Fix an integer $n\in \mathbb{N}$ and a polynomial $p(t) \in \mathbb{Q}[t]$. 
Then, there is a proper Deligne--Mumford stack $\sF_{n,p,1}$, with projective coarse moduli space, parametrizing $\bQ$-stable ($\bQ$-Gorenstein) pairs $(X;D)$ of dimension $n$ and with polynomial $p(t)$ and reduced boundary.
\end{theorem}

In the statement of Theorem \ref{thm intro 1}, the subscript 1 in $\sF_{n,p,1}$ means that $D$ is a reduced divisor.

While it is very natural to consider the divisor $D$ with its reduced structure, the framework of the Minimal Model Program highlights the importance of the use of fractional coefficients for the divisor $D$.
For example, in the case of curves, one can replace $\overline{\cM}_{g,n}$ with a weighted version, where the markings can attain any fractional value in $(0,1]$.
This was accomplished by Hassett in \cite{Has03}, and this approach leads to different compactifications of
$\cM_{g,n}$.

It turns out, however, that the construction of the higher dimensional analogs of these weighted moduli spaces is very delicate.
Among the many difficulties, the definition of a suitable notion of family of boundary divisors represents a major problem, see \cite{kol_new}*{Ch. 4}.
Nevertheless, over the past decade, there has been significant progress in the development of a moduli theory for higher dimensional stable pairs (see \cites{ Ale94, Kol13, kol08, HX13, KP17, kol_new, HMX }), and this last missing piece has finally been settled by Koll\'ar in \cite{kol19s}.
In \emph{loc. cit.} a rather subtle refinement of the flatness condition is introduced, leading to the ultimate notion of family of divisors, that in turn gives a satisfactory treatment of a moduli functor of families of stable pairs, admitting divisors with arbitrary coefficients.

The difficulty in defining a suitable notion of family of boundary divisors is due to the fact that, in general, a deformation of a pair $(X,D)$ cannot be reduced to a deformation of the total space $X$ and a deformation of the divisor $D$.
This na\"ive expectation is only true if the coefficients of $D$ are all strictly greater than $\frac{1}{2}$ (see \cite{Kol14}), in which case a deformation of $(X,D)$ over a base curve induces a flat deformation of $D$.
On the other hand, if smaller coefficients are allowed, the situation becomes much more subtle.
For instance, as showed by Hassett (see \S~\ref{hassett example} for details), by allowing the coefficient $\frac{1}{2}$, it is possible to define a family of stable surface pairs $(\mathcal{X},\mathcal{D}) \rar \mathbb A ^1$ such that the flat limit of $\mathcal{D}$ acquires an embedded point along the special fiber.

One advantage of our formalism is that it allows overcoming the aforementioned difficulties.
That is, our setup easily generalizes to the case of any fractional coefficients, preventing the above-mentioned pathologies of families of divisors, and leading to a more transparent definition of the moduli functor in Definition \ref{def families A+}.
Indeed, we generalize $\sF_{n,p,1}$ to an analogous moduli functor $\sF_{n,p,I}$ in Definition \ref{def families A+}, allowing arbitrary coefficients on the divisor.
The main feature of Definition \ref{def families A+} is that the family of boundaries is characterized as a flat, proper, and relatively $S_1$ morphism over the base, thus significantly simplifying the definition of the classical moduli functor.
The $S_1$ condition prevents the existence of embedded points, while the polynomial $p(t)$ controls the relevant invariants of the varieties, preventing them from jumping as in Hassett's example. This leads to the following generalization of Theorem \ref{thm intro 1}:

\begin{theorem} \label{thm intro 1 on steroids}
Fix an integer $n\in \mathbb{N}$, a finite subset $I\subseteq (0,1] \cap \mathbb{Q}$, and a polynomial $p(t) \in \mathbb{Q}[t]$.
Assume that $I$ is closed under sum: that is, if $a,b \in I$ and $a+b \leq 1$, then $a+b \in I$.
Then there is a proper Deligne--Mumford stack $\sF_{n,p,I}$ with projective coarse space, parametrizing $\bQ$-stable pairs of dimension $n$ and coefficients in $I$.
\end{theorem}

After the completion of this work, we learnt that Koll\'ar has also developed an analogous machinery allowing pairs $(X;D_1,...,D_k)$ where the $k$ boundary divisors have coefficients that can be independently perturbed; see \cite{kol_new}*{\S~8.3}.

In \S~\ref{section morphism} we relate our moduli functor with the moduli functor of stable pairs, and as the main application of the theory of $\bQ$-stable pairs, we obtain a simpler proof of the projectivity of the moduli space of stable pairs, originally proved by Kov\'acs and Patakfalvi \cite{KP17}.

\begin{theorem}[Corollary \ref{cor proj moduli}] \label{thm intro 2}
Consider a proper DM stack $ \cK_{n,v,I}$ that satisfies the following two conditions:
\begin{enumerate}
    \item for every normal scheme $S$, the data of a morphism $f \colon S\to \cK_{n,v,I}$ is equivalent to a stable family of pairs $q \colon (\cY,\cD)\to B$ with fibers of dimension $n$, volume $v$ and coefficients in $I$; and
    \item there is $m_0\in \mathbb{N}$ such that, for every $k \in \mathbb{N}$, there is a line bundle $\cL_k$ on $\cK_{n,v,I}$ such that, for every morphism $f$ as above, $f^*\cL_k \cong \det(q_*(\omega_{\cY/B}^{[km_0]}(km_0\cD)))$.
\end{enumerate}
Then, the coarse moduli space of $\cK_{n,v,I}$ is projective. 
\end{theorem}

As observed in \cite{KP17}*{\S~1.1}, the approach of \cite{Has03} for proving projectivity of these moduli spaces cannot be adapted to higher dimensions, as certain sheaves are no longer functorial with respect to base change.
For this reason, Kov\'acs and Patakfalvi develop a refinement of Koll\'ar's ampleness lemma.
On the other hand, in the setup of $\bQ$-stable pairs, all the needed sheaves remain functorial with respect to base change.
Indeed, flatness and the $S_1$ condition guarantee that all notions of pull-back agree.
Thus, we can follow Hassett's strategy and directly apply \cite{Kol90}.
In this way, the projectivity part of Theorem \ref{thm intro 1 on steroids} is established;
then, to deduce Theorem \ref{thm intro 2}, it suffices to show that the moduli space of $\bQ$-stable pairs naturally admits a \emph{finite} morphism to the moduli space of stable pairs, and the needed polarization descends with this morphism.

\subsection{Structure of the paper}
\label{struttura}
The first part of this work is devoted to developing the notion of $\bQ$-stable pair and extending several statements from pairs to $\bQ$-stable pairs.
In particular, in \S~\ref{preliminaries} we set the key definitions and properties, while in \S~\ref{section_boundedness} we extend the boundedness results of \cite{HMX18} to the context of $\bQ$-stable pairs.

In \S~\ref{section functor} and \S~\ref{section_properness}, we analyze the moduli functor $\sF_{n,p,I}$.
In particular, in \S~\ref{section functor} we show that $\sF_{n,p,I}$, which a priori is only a category fibered in groupoids over $\mathrm{Sch}/k$, is a Deligne--Mumford stack.
Then, in \S~\ref{section_properness} we show that $\sF_{n,p,I}$ is proper.
Thus, \S~\ref{section functor} and \S~ \ref{section_properness} settle Theorem \ref{thm intro 1 on steroids}, except for the projectivity part.

In \S~\ref{section morphism}, we analyze the notion of family of $\bQ$-stable pairs when the base is reduced.
Under these assumptions, we show that the existence of a relative good minimal model is determined by a fiberwise condition.

In \S~\ref{section projectivity ksba} we conclude our work.
We use the ampleness lemma to show that the coarse space of $\sF_{n,p,I}$ is projective, and we use the results from \S~\ref{section morphism} to descend the analogous statement to the moduli space of stable pairs.

\subsection{Connections with other works}
One of the main difficulties to extend the moduli theory from stable varieties to stable pairs lays in the behavior of the boundary divisor in families.
Indeed, given a pair $(X,D)$, $K_X$ and $D$ need not be $\mathbb Q$-Cartier.
In this case, we cannot expect a deformation of $(X,D)$ to induce a flat deformation of $D$, so one cannot define the moduli problem by simply requiring that the divisor varies flatly in families.
A final and satisfactory notion for defining families of divisors has been achieved in \cite{kol19s} with the notion of \textit{K-flatness}.

Over the years, there have been attempts to overcome these difficulties in some special situations where the deformations of $(X,D)$ do induce flat deformations of $D$.
In \cites{Ale06,Ale15}, Alexeev considered pairs $(X,\Delta=\sum_{i=1}^l a_i D_i)$, where each $D_i$ is prime and the coefficients $a_i$ are irrational and $\mathbb Q$-linearly independent.
This assumption can be thought as a ``general choice'' of the coefficients of $D$, as opposed to having rational coefficients.
Indeed, many of the examples, it can be showed that, as we vary the coefficients of the boundary, the behavior of the moduli space is locally constant, leading to a chamber decomposition of some appropriate polytope contained in $[0,1]^l$.
These chambers have rational vertices, determined by the conditions of $(X,D)$ being semi-log canonical and of $K_X+D$ being ample.
For more details on this topic, we refer to \cite{ABIP}, and some examples of this phenomena are illustrated in \cites{Ale15, inc20, AB21}.

In a different direction, Koll\'ar showed in \cite{Kol14} that, if the coefficients of $D$ are strictly greater than $\frac{1}{2}$, deformations of $(X,D)$ lead to flat deformations of $D$.
Thus, even though $D$ is not necessarily $\mathbb Q$-Cartier, this allows for a satisfactory moduli theory that does not rely on the latest developments in \cite{kol19s}.

In this paper, we work with rational coefficients and impose that the boundary $D$ is $\mathbb Q$-Cartier.
As we do not require the coefficients to be irrational, in order to retain the properness of the moduli problem, we trade the ampleness of $K_X+D$ for the flatness of the deformations of $D$.
Thus, while the deformations of the boundary are easier to understand than the general framework of \cite{kol19s}, $K_X+D$ is in general only big and semi-ample.

In \cite{kol_new}*{\S~8.3}, Koll\'ar presents a more general version of the approach pursued in this work.
In particular, he considers arbitrary coefficients of $D$ and allows for independent perturbations of different components of $D$.
In this way, he reconciles Alexeev's work with our work, thus showing that the case of ``$D$ with rational coefficients and $K_X+D$ big and semi-ample'' we consider can be thought of as a limit case of the ``general coefficients with $K_X+D$ ample'' considered by Alexeev.

Similarly, we remark that many of the ideas in the current paper fit in the general formalism of stable minimal models, developed by Birkar (see for example \cite{birkar2021boundedness}), culminating in \cite{birkar2022moduli}, where he develops a moduli theory of varieties and pairs of non-negative Kodaira dimension.
We refer the reader to Definition~\ref{def_p-pair} and \cite{birkar2021boundedness}*{Definition 1.8} for the similarities between the definitions of $\bQ$-stable pair and stable minimal models.
In particular, we observe that both definitions entail a stability condition prescribed by a polynomial, which guarantees the boundedness of the moduli problem, thus taking inspiration from ideas of Viehweg \cite{Vie95}.

\subsection*{Acknowledgements} 
We thank J\'anos Koll\'ar for pointing to us a mistake in an earlier version of this work, for sharing the latest draft of his book on moduli, for providing helpful comments to improve our work, and for providing feedback on an eralier version of this work.
We thank Jarod Alper, Dori Bejleri, Christopher Hacon, S\'andor Kov\'acs, Joaqu\'in Moraga, and Zsolt Patakfalvi for helpful discussions.
We thank the anonymous referees for helpful suggestions.

\section{Preliminaries} \label{preliminaries}

\subsection{Terminology and conventions}
\label{term.subs}
Throughout this paper, we will work over the field of complex numbers.
For the standard notions in the Minimal Model Program (MMP) that are not addressed explicitly, we direct the reader to the terminology and the conventions of~\cite{KM98}.
Similarly, for the relevant notions regarding non-normal varieties, we direct the reader to \cite{Kol13}.
A variety will be an integral separated scheme of finite type over $\bC$.
A birational morphism between non-normal schemes will be a morphism $f \colon X \to Y$ with a dense open subset $U\subseteq Y$ such that $f^{-1}(U)$ is dense and $f^{-1}(U)\to U$ is an isomorphism.

\subsection{Contractions}
A \emph{contraction} is a projective morphism $f\colon X \rar Z$ of quasi-projective varieties with $f_\ast  \O X. = \O Z.$. 
If $X$ is normal, then so is $Z$.

\subsection{Divisors}
Let $\mathbb{K}$ denote $\zz$, $\qq$, or $\rr$. We say that $D$ is a \emph{$\mathbb{K}$-divisor} on a variety $X$ if we can write $D = \sum \subs i=1. ^n d_i P_i$ where $d_i \in \mathbb{K} \setminus \{ 0\}$, $n \in \nn$ and the $P_i$ are prime Weil divisors on $X$ for all $i=1, \ldots, n$. 
We say that $D$ is $\mathbb{K}$-Cartier if it can be written as a $\mathbb{K}$-linear combination of $\zz$-divisors that are Cartier.
The \textit{support} of a $\mathbb{K}$-divisor $D=\sum_{i=1}^n d_iP_i$ is the union of the prime divisors appearing in the formal sum $\mathrm{Supp}(D)= \sum_{i=1}^n P_i$.

In all of the above, if $\mathbb{K}= \zz$, we will systematically drop it from the notation.

Given a prime divisor $P$ in the support of $D$, we will denote by $\mu_P (D)$ the coefficient of $P$ in $D$.
Given a divisor $D = \sum \mu_{P_i}(D) P_i$ on a normal variety $X$, and a morphism $\pi \colon X \to Z$, we define
\[
D^v \coloneqq \sum_{\pi(P_i) \subsetneqq Z} \mu_{P_i}(D) P_i, \
D^h \coloneqq \sum_{\pi(P_i) = Z} \mu_{P_i}(D) P_i.
\]

Let $D_1$ and $D_2$ be divisors on $X$.
We write $D_1 \sim_{\mathbb K ,Z} D_2$ if there is a $\mathbb{K}$-Cartier divisor $L$ on $Z$ such that $D_1 - D_2 \sim _{\mathbb K}f^\ast L$.
Equivalently, we may also write $D_1 \sim_{\mathbb{K}} D_2/Z$, or $D_1 \sim_{\mathbb{K}} D_2$ over $Z$.
If $\mathbb{K}=\zz$, we omit it from the notation.
Similarly, if $Z= \mathrm{Spec}(k)$, where $k$ is the ground field, we omit $Z$ from the notation.

Let $\pi \colon X \rar Z$ be a projective morphism of normal varieties.
Let $D_1$ and $D_2$ be two $\mathbb{K}$-divisors on $X$.
We say that $D_1$ and $D_2$ are numerically equivalent over $Z$, and write $D_1 \equiv D_2/Z$, if $D_1 . C = D_2 . C$ for every curve $C \subset X$ such that $\pi(C)$ is a point.
In case the setup is clear, we just write $D_1 \equiv D_2$, omitting the notation $/Z$.

\subsection{Non-normal varieties and pairs} There are two important generalizations of the notion of normal variety.
\begin{Def}
An $S_2$ scheme is called \emph{demi-normal} if its codimension 1 points are either regular or nodal.\end{Def}
Roughly speaking, the notion of demi-normal schemes allows extending the notion of log canonical singularities to non-normal varieties, allowing for a generalization of the notion of stable curve to higher dimensions.
We refer to \cite{kol_new}*{\S~10.8} for more details.

\begin{Def}A finite morphism of schemes $X' \rar X$ is called a \emph{partial seminormalization} if $X'$ is reduced and, for every point $x \in X$,  the induced map $k(x) \rar k(\mathrm{red}(g^{-1}(x)))$ is an isomorphism.
There exists a unique maximal partial seminormalization, which is called \emph{the seminormalization} of $X$.
A scheme is called \emph{seminormal} if the seminormalization is an isomorphism.\end{Def}

This is an auxiliary lemma, it is probably well known. We include it for completeness.
We refer to \cite{kol_new}*{\S~10.8} for the details about seminormality.

\begin{Lemma}\label{lemma:connected:fivers:vs:cohomol:connected:fibers}
Let $p \colon X\to Y$ be a proper surjective morphism with connected fibers, with $Y$ seminormal and $X$ reduced.
Then $p_*\cO_X = \cO_Y$.
\end{Lemma}

\begin{Remark}\label{remark:cohom:connected:implies:one:can:take:global:sections:after:pull:back}
Observe that, in situation of Lemma \ref{lemma:connected:fivers:vs:cohomol:connected:fibers}, by the projection formula it follows that, for any line bundle $L$ on $Y$, we have $H^0(Y,L) =H^0(X,p^*L)$.
\end{Remark}

\begin{proof}
We can take the Stein factorization $X\xrightarrow{q} Z\xrightarrow{a} Y$.
Since $X$ is reduced, then $Z$ is reduced.
Observe that $q_*\cO_X = \cO_Z$, so the desired result follows if we can show that $a$ is an isomorphism.

Since the composition $p=a \circ q$ has connected fibers and $q \colon X\to Z$ is surjective, then $a \colon Z\to Y$ has connected fibers.
Since $a$ is finite, it is injective.
It is also surjective since $p$ is surjective, so $a$ is a bijection.
Then, we can take the seminormalization $Z^{sn}\to Z$ and consider the composition $Z^{sn}\to Z \to Y$.
This is a bijective morphism since it is a composition of bijections, and it is proper since $a$ is proper and $Z^{sn}\to Z$ is proper.
Then, the composition $Z^{sn} \to Y$ is an isomorphism, since the source and the target are seminormal.

Now, on the topological space given by $Z$ we have the following morphsims of sheaves:
$$\cO_Y\xrightarrow{a^\#} \cO_Z\xrightarrow{b} \cO_{Z^{sn}}.$$
The composition is surjective, so $b$ is surjective. Moreover, since $Z$ is reduced, the map $\cO_Z\xrightarrow{b} \cO_{Z^{sn}}$ is also injective. Then it is an isomorphism, so $Z\cong Z^{sn}\cong Y$.
\end{proof}

\subsection{Divisorial sheaves}\label{divisorial sheaves}
Throughout this section, $X$ will be $S_2$ and reduced. We begin this subsection with the following definition:

\begin{Def}Let $X$ be a scheme.
A sheaf $\mathfrak F$ on $X$ is called \emph{divisorial sheaf} if it is $S_2$ and there is a closed subscheme $Z \subset X$ of codimension at least 2 such that $\mathfrak F | \subs X \setminus Z.$ is locally free of rank 1.\end{Def}

\begin{Def}[\cite{HK04}*{\S~3}]Let $f \colon X\to S $ a flat morphism of schemes, and $\cE$ a coherent sheaf on $X$. We say that $\cE$ is relatively $S_n$ if it is flat over $S$ and its restriction to each fiber is $S_n.$
\end{Def}

Now, let $X$ be a demi-normal scheme, and let $\mathrm{Weil}^*(X)$ denote the subgroup of $\mathrm{Weil}(X)$ generated by the prime divisors that are not contained in the conductor of $X$.
Then, there is an identification between $\mathrm{Weil}^*(X)/\sim$ and the group of divisorial sheaves, where $\sim$ denotes linear equivalence.
The identification is defined as follows.
By the definition of $\mathrm{Weil}^*(X)$ and the demi-normality of $X$, for every element $B \in \mathrm{Weil}^*(X)$, there is a closed subset $Z \subset X$ of codimension at least 2 such that $B| \subs X \setminus Z.$ is a Cartier divisor.
Then, the corresponding divisorial sheaf is defined as $j_* \O X \setminus Z. (B| \subs X \setminus Z.)$, where we have $j \colon X \setminus Z \rar X$.

Consider a flat family $f \colon X\to B$ with $S_2$ fibers, and assume that there is an open subset $i \colon U\subseteq X$ such that $U_b$ is big for every $b\in B$.
Then, for every locally free sheaf $\cG$ on $U$, we can consider $i_*\cG$ on $X$.
From \cite{HK04}*{Corollary 3.7}, this is a reflexive sheaf on $X$ (a priori, it is not $S_2$ relatively to $B$, see \S~\ref{hassett example}).

If $X$ is demi-normal, its canonical sheaf is a divisorial sheaf, as $X$ is Gorenstein in codimension 1.
By the above identification, we can then choose a canonical divisor $K_X$ such that $\O X. (K_X) \cong \omega_X$.
Observe that this construction can be carried out in families.
Indeed, if $f \colon X\to B$ is a flat morphism with demi-normal fibers of dimension $n$, there is an open locus $i \colon U\subseteq X$ that has codimension 2 along each fiber, on which $f$ is Gorenstein.
Then we can define $\omega_{X/B} \coloneqq i_*\omega_{U/B}$.
Observe that this agrees with the $(-n)$-th cohomology of the relative dualizing complex.
Indeed, the latter is $S_2$ (see \cite{LN18}*{\S~5}).

Let $X$ be demi-normal, and consider two divisorial sheaves $L_1$ and $L_2$.
Then, their \emph{reflexive tensor product} is defined as $L_1 \hat{\otimes} L_2 \coloneqq (L_1 \otimes L_2)^{**}$ and it is a divisorial sheaf itself.
If we have $L_1 \cong \O X. (D_1)$ and $L_2 \cong \O X. (D_2)$, we have $L_1 \hat{\otimes} L_2 \cong \O X. (D_1+D_2)$.
The \emph{$m$-fold reflexive power $L^{[m]}$} is defined as the $m$-fold self reflexive tensor product.

For more details, we refer to \cite{kol_new}*{\S~3.3}, \cite{Kol13}*{5.6} and \cite{HK04}.

\subsection{Boundedness}
Let $\mathfrak{D}$ be a set of projective pairs.
Then, we say that $\mathfrak{D}$ is {\it log bounded} (resp. {\it log birationally bounded}) if there exist a variety $\mathcal{X}$, a reduced divisor $\mathcal{B}$ on $\mathcal{X}$, and a projective morphism $\pi \colon \mathcal{X} \rar T$, where $T$ is of finite type, such that $\mathcal{B}$ does not contain any fiber of $\pi$, and, for every $(X,B) \in \mathfrak{D}$, there are a closed point $t \in T$ and a morphism (resp. a birational map) $f_t \colon \mathcal{X}_t \rar X$ inducing an isomorphism $(X,\Supp(B)) \cong (\mathcal{X}_t,\mathcal{B}_t)$ (resp. such that $\Supp(\mathcal B _t)$ contains the strict transform of $\Supp(B)$ and all the $f_t$ exceptional divisors).

A set of projective pairs $\mathfrak{D}$ is said to be {\it strongly log bounded} if there is a quasi-projective pair $(\mathcal{X},\mathcal{B})$ and a projective morphism $\pi \colon \mathcal{X} \rar T$, where $T$ is of finite type, such that $\Supp(\mathcal{B})$ does not contain any fiber of $\pi$, and for every $(X,B) \in \mathfrak D$, there is a closed point $t \in T$ and an isomorphism $f \colon X \rar \mathcal{X}_t$ such that $f_*B=\mathcal{B}_t$.

A set of projective pairs $\mathfrak{D}$ is {\it effectively log bounded} if it is strongly log bounded and we may choose a bounding pair $\pi \colon (\mathcal{X},\mathcal{B}) \rar T$ such that, for every closed point $t \in T$, we have $(\mathcal{X}_t,\mathcal{B}_t) \in \mathfrak D$.

\subsection{Index of a set}\label{indice}
Given a finite subset $I \subseteq \mathbb{Q}^n$, we define the \emph{index} of $I$ to be the smallest positive rational number $r$ such that $rI\subseteq \mathbb{Z}^n$.

\subsection{Stable pairs and $\bQ$-stable pairs} \label{p-pairs section}
Let $(X,D)$ denote a projective semi-log canonical pair, where $D$ has coefficients in $\qq$.
We say that $(X,D)$ is a \emph{stable pair} if $K_X + D$ is ample.
\begin{Def}\label{def_p-pair}
Consider a polynomial $p(t)\in \mathbb{Q}[t]$ and a set $I\subseteq (0,1]$.
A \emph{$\bQ$-pair} with polynomial $p(t)$ and coefficients in $I$ is the datum of a semi-log canonical pair $(X,\Delta)$ and a $\qq$-Cartier $\qq$-divisor $D$ on $X$ satisfying the following properties:
\begin{enumerate}
    \item $(X,D+\Delta)$ is semi-log canonical;
    \item $p(t) = (K_X+tD+\Delta)^{\dim(X)}$; 
    \item there is a stable pair $(X^c,\Delta^c+D^c)$ with a birational contraction $\pi \colon X\to X^c$ such that $\pi^*(K_{X^c}+\Delta^c + D^c) = K_{X}+\Delta + D$, $\pi_*D=D^c$ and $\pi_*\Delta=\Delta^c$; and
    \item the coefficients of $D$ and $\Delta$ are in $I$.
\end{enumerate}
If moreover there is $\epsilon_0>0$ such that for every $0<\epsilon \le\epsilon_0$ the pair $(X,(1-\epsilon)D+\Delta)$ is stable, the $\bQ$-pair $(X,\Delta;D)$ is called a \emph{$\bQ$-stable pair}. Finally, we will call $(X^c,\Delta^c+D^c)$ the \emph{canonical model} of $(X,\Delta; D)$.\end{Def}

\begin{Remark}
If $(X,\Delta;D)$ is a pair that satisfies points (1), (2) and (4) of Definition \ref{def_p-pair} and such that there is $\epsilon_0>0$ such that for every $0<\epsilon \le\epsilon_0$ the pair $(X,(1-\epsilon)D+\Delta)$ is stable, then one can prove that point (3) of Definition \ref{def_p-pair} is equivalent to $K_X+D+\Delta$ being basepoint free.
\end{Remark}

For brevity, we denote the datum of a $\bQ$-pair by $(X,\Delta;D)$, where the polynomial $p(t)$ and the set of coefficients $I$ are omitted in the notation.
In case $\Delta=0$, we then write $(X;D)$.
\begin{Remark}
For example, if $(X,0;D)$ is a semi-log canonical $\bQ$-stable pair, and $\nu \colon X^n\to X$ is the normalization of $X$ with conductor $\Delta$, then $(X^n,\Delta;\nu_*^{-1}D)$ is a $\bQ$-stable pair.
\end{Remark}

\begin{Notation}\label{notation_Dsc}
If $I$ is a finite set and $r$ is its index, given a $\bQ$-pair $(X,\Delta;D)$ with polynomial $p(t)$ and coefficients in $I$ we denote by $D^{sc}$ as the subscheme of $X$ defined by the reflexive sheaf of ideals $\O X.(-rD)$.
\end{Notation}

\begin{Remark}\label{remark_KX+D+delta_is_nef}
Observe that if $(X,\Delta;D)$ is a $\bQ$-stable pair, the $\qq$-divisor $K_X+D+\Delta$ is nef since it is limit of ample $\qq$-divisors.
Furthermore, $K_X+D+\Delta$ is big, as it is the sum of an ample $\qq$-divisor and an effective $\qq$-divisor.
\end{Remark}

\begin{Lemma} \label{lemma_convex_combination}
Let $X$ be a projective variety, and let $D_1$ and $D_2$ be two nef $\bQ$-Cartier $\bQ$-divisors.
Assume that for some $t \in (0,1)$, the divisor $tD_1 + (1-t)D_2$ is ample.
Then, $cD_1 + (1-c)D_2$ is ample for every $c \in (0,1)$.
\end{Lemma}

\begin{proof}
Since ampleness is an open condition, we may assume that $t \in \mathbb Q$.
Then, using convex combinations, the claim follows from Kleiman's criterion and the denseness of $ \mathbb{Q}$ in $\mathbb{R}$.
\end{proof}

\begin{Lemma} \label{ample model p-pair}
Let $(X,\Delta;D)$ be a $\bQ$-stable pair, and assume that $f \colon X \rar Y$ is its canonical model.
Then, $\mathrm{Ex}(f) \subset \supp(D)$.
\end{Lemma}

\begin{proof}
By assumption, there exists $\epsilon>0$ such that $(X,(1-\epsilon)D + \Delta)$ is a stable pair.
Thus, if $C$ is an irreducible curve that is not contained in $\supp(D)$, we have that
\[
(\K X. + D + \Delta).C = (\K X. + (1-\epsilon)D + \Delta).C + \epsilon D .C >  0,
\]
since the first summand is positive by ampleness and the second is non-negative as $C$ is not contained in $\supp(D)$.
Thus, the curves contracted by $f$ are contained in $\supp(D)$ and the claim follows.
\end{proof}
The following lemma is the main technical tool that we will use in \S~\ref{section_boundedness} and \S~\ref{section_properness}:
\begin{Lemma}\label{lemma_bound_epsilon_lc_case}
Let $n$ and $M$ be natural numbers.
Then, there exists $\epsilon_0 > 0$, only depending on $n$ and $M$, such that the following holds.
Let $(Y,\Delta_Y;D_Y)$ be a normal $\bQ$-stable pair of dimension $n$, and let $\pi \colon Y \rar X$ be its canonical model $(X,D_X+\Delta_X)$.
Further assume that the Cartier index of $K_X + D_X + \Delta_X$ is less than $M$.
Then, for every $0 < \epsilon < \epsilon_0$, the pair $(Y,(1-\epsilon)D_Y+\Delta_Y)$ is stable.
\end{Lemma}

\begin{proof}
By definition of $\bQ$-stable pair and of canonical model,
for every $\pi$-exceptional curve $C$, we have $(K_Y+D_Y+\Delta_Y).C=0$ and $(K_Y+(1-\epsilon)D_Y+\Delta_Y).C>0$ for $0 < \epsilon \ll 1$.
For every such curve $C$, the function $t\mapsto (K_Y+tD_Y+\Delta_Y).C$ is linear and not identically 0, so it has at most one zero. 
In particular, for every $\epsilon >0$ we have 
\begin{equation} \label{eq_intersection0}
(K_Y+(1-\epsilon)D_Y+\Delta_Y).C>0.    
\end{equation}
Hence, if we choose $\epsilon =1$, we have
\begin{equation} \label{eq_intersection1}
(K_Y+\Delta_Y).C>0
\end{equation}
for every $\pi$-exceptional curve.

By the same argument, for every curve $C'$ such that $(K_Y + \Delta_Y).C' \geq 0$, it follows that
\begin{equation} \label{eq_mori_cone}
    (K_Y  + (1-\eta) D_Y + \Delta_Y).C' >0
\end{equation}
for every $\eta \in (0,1)$.
Now, define
\[
a\coloneqq \frac{1}{2M},\quad \epsilon_0 \coloneqq \displaystyle{\frac{a}{2n+a+1}}.
\]

Since $K_X+D_X+\Delta_X$ is ample with Cartier index bounded by $M$, for every curve $\Gamma$ that is not $\pi$-exceptional,
by the projection formula, we have
\begin{equation} \label{eq_intersection2}
\frac{1}{2M} = a<(K_X+D_X+\Delta_X).p_*\Gamma = (K_Y+D_Y+\Delta_Y).\Gamma.
\end{equation}
From \cite{Fuj11}*{Theorem 1.4}, for every $(K_Y+\Delta_Y)$-negative extremal ray $R$, we may find a curve $\Gamma$ generating $R$ such that
\begin{equation} \label{eq_bdd_negative_curves}
-2n\le (K_Y+\Delta_Y).\Gamma<0.
\end{equation}
By \eqref{eq_intersection1}, any such curve is not $\pi$-exceptional.
Then, by \eqref{eq_intersection2} and \eqref{eq_bdd_negative_curves}, for any such $\Gamma$, we have
\begin{equation} \label{eq_extremal_curves}
    (K_Y + (1-\epsilon_0)D_Y + \Delta_Y).\Gamma= (1-\epsilon_0)(K_Y + D_Y + \Delta_Y).\Gamma + \epsilon_0 (K_Y + \Delta_Y).\Gamma>(1-\epsilon_0)a-2n\epsilon_0 > 0.
\end{equation}

Now, consider the cone of curves $\overline{NE}(Y)$ and its decomposition given by the cone theorem associated to the pair $(Y,\Delta_Y)$ \cite{Fuj11}*{Theorem 1.4}.
Then, by \eqref{eq_mori_cone}, we have that $K_Y + (1-\epsilon_0)D_Y + \Delta_Y$ is positive on $\overline{NE}(Y)_{K_Y + \Delta_Y \geq 0}$.
Thus, every $(K_Y + (1-\epsilon_0)D_Y + \Delta_Y)$-negative extremal ray is also a $(K_Y +  \Delta_Y)$-negative extremal ray.
Then, by \eqref{eq_extremal_curves}, we have that $(K_Y + (1-\epsilon_0)D_Y + \Delta_Y)$ is positive on $R$.
Thus, by the cone theorem \cite{Fuj11}*{Theorem 1.4}, the log canonical pair $(Y,(1-\epsilon_0)D_Y + \Delta_Y)$ has no negative extremal rays; thus, $K_Y + (1-\epsilon_0)D_Y + \Delta_Y$ is nef.

Now, by the definition of $\bQ$-stable pair, some convex combination of $K_Y+(1-\epsilon_0)D_Y+\Delta_Y$ and $K_Y+D_Y+\Delta_Y$ is ample.
Then, the claim follows by Lemma \ref{lemma_convex_combination}.
\end{proof}

\begin{Lemma}\label{lemma_bound_for_epsilon}
Fix an integer $n\in \mathbb{N}$, a volume $v\in \mathbb{Q}_{>0}$, and a finite subset $I\subseteq (0,1] \cap \mathbb{Q}$.
There is $0<\epsilon_0$, only depending only on $n$ and $v$, such that, for every $\bQ$-stable pair $(Y,\Delta_Y;D_Y)$ of dimension $n$, coefficients in $I$, and polynomial $p(t)$ with $p(1) = v$, the pair $(Y,(1-\epsilon)D_Y+\Delta_Y)$ is stable for every $0<\epsilon < \epsilon_0$.
\end{Lemma}

\begin{proof}
Without loss of generality, we may assume from now on that $1 \in I$.
Let $Y^n\to Y$ be the normalization of $Y$.
Let $(Y^n,\Delta_Y^n)$ be the pair induced by $(Y,\Delta_Y)$, and let $D_Y^n$ denote the pull-back of $D_Y$ to $Y^n$.
Here, $\Delta_Y^n$ consists of the sum of the divisorial part of the preimage of $\Delta_Y$ and the conductor.
Notice that $Y^n$ possibly has more than one connected component.
Then, we have:
\begin{enumerate}
     \item $(Y^n,D_Y^n+\Delta_Y^n)$ has still coefficients in $I$;
     \item $K_{Y^n} + D_Y^n+\Delta_Y^n$ is semi-ample and big; and
     \item $(Y^n,\Delta_Y^n;D_Y^n)$ is a $\bQ$-stable pair.
\end{enumerate}
From \cite{HMX14}*{Theorem 1.3},
there are finitely many possibilities to write $v$ as $v = \sum v_i$, where $v_i>0$ is the volume of a log canonical pair of general type with coefficients in $I$.
So from (1), there is a finite set $\sS=\{v_1,...,v_k\}$ such that the volume of each connected component of $Y^n$ is in $\sS$. 

For a log canonical pair of general type $(Z,D)$,
the canonical model $(Z^c,D^c)$ 
is such that $\vol(Z,D) = \vol(Z^c,D^c)$.
In particular, it follows from (2) that the irreducible components of $(Y^n,D_Y^n+\Delta_Y^n)$ admit a canonical model, and the volumes of these models are contained in $\cS$ as well.
From \cite{HMX}*{Theorem 1.1}, there is an $M>0$ such that every stable pair with coefficients in $I$, dimension $n$, and volume in the finite set $\sS$ has Cartier index less than $M$.
Since we can check ampleness after passing to the normalization, and the Cartier indexes of each irreducible component of the normalization are bounded, the thesis follows from Lemma \ref{lemma_bound_epsilon_lc_case}.
\end{proof}

\begin{Lemma} \label{lift stable pair slc}
Let $(Y,D_Y)$ be a stable pair.
Then, there exists a $\bQ$-stable pair $(X;D)$ having $(Y,D_Y)$ as canonical model.
\end{Lemma}
 Observe in particular that from Lemma \ref{ample model p-pair}, the morphism $\pi \colon X\to Y$ induces an isomorphism at the generic point of the codimension one singular locus of $X$ and $Y$, as those are points not contained in $D$, and $\pi$ is an isomorphism away from the exceptional locus. So $X$ is normal if and only if its codimension one singular locus is empty and if and only if the one of $Y$ is empty.

\begin{proof}
We consider the semi-canonical modification $\pi \colon X \rar Y$ of $(Y,0)$, in the sense of \cite{Fuj15}.
We observe that, in order to consider a semi-canonical modification, $K_Y$ does not need to be $\qq$-Cartier.
Such modification exists by \cite{Fuj15}*{Theorem 1.1} and the fact that $Y$ is demi-normal.
By \cite{Fuj15}*{Definition 2.6}, $\pi \colon X \rar Y$ has the following properties:
\begin{itemize}
    \item $\pi$ is an isomorphism around every generic point of the double locus of $X$;
    \item this procedure is compatible with taking the normalizations of $X$ and $Y$, see \cite{Fuj15}*{Lemma 3.7.(2)}.
    In particular, $X$ is normal if so is $Y$, and, in general, $\pi$ establishes a bijection between irreducible components of $X$ and $Y$; and
    \item $K_X$ is $\qq$-Cartier and $\pi$-ample.
\end{itemize}
Now, set $K_X + D = \pi^*(K_Y+D)$.
Since $K_Y+D$ is ample, for $0 < \epsilon \ll 1$, we have that $\epsilon K_X + (1-\epsilon)(K_X+D)=K_X + (1-\epsilon)D$ is ample.
Since both $K_X$ and $K_X+D$ are $\qq$-Cartier, then so is $D$.
Lastly, as $X$ and $(Y,D_Y)$ are semi-log canonical, to conclude it suffices to show that $D \geq 0$.

Now, let $(X^\nu,D^\nu + \Delta^\nu)$ and $(Y^\nu,D_Y^\nu + \Delta^\nu_Y)$ denote the respective normalizations, where $\Delta^\nu$ and $\Delta_Y^\nu$ denote the double loci.
Thus, it suffices to show that $D^\nu \geq 0$.
Since this can be checked by considering one irreducible component of $X^\nu$ at the time, by abusing notation, we may assume that $X^\nu$ and $Y^\nu$ are irreducible.
By construction, we have
\[
\K X^\nu. + (1-\epsilon)D^\nu + \Delta^\nu \sim_{\qq}-\epsilon D^\nu \sim_{\qq} \epsilon(\K X^\nu. + \Delta^\nu) /Y^\nu,
\]
and $\K X^\nu. + \Delta^\nu$ is relatively ample over $Y^\nu$.
Then, since we have $D^\nu_Y \geq 0$, by the negativity lemma \cite{KM98}*{Lemma 3.39}, it follows that $D^\nu \geq 0$.
This concludes the proof.
\end{proof}

\subsection{Families of pairs} 
We recall the main definitions of families of pairs from \cite{kol_new}*{Ch. 4} and \cite{kol19s}. A \emph{family of pairs} $f \colon (X,D) \rar S$ over a reduced base $S$ is the datum of a morphism $f \colon X \rar S$ and an effective $\qq$-divisor $D$ on $X$, such that the following conditions hold:

\begin{itemize}
    \item $f$ is flat with reduced fibers of pure dimension $n$;
    \item the fibers of $\Supp(D) \rar S$ are either empty or of pure dimension $n-1$; and
    \item $f$ is smooth at the generic points of $X_s \cap \Supp(D)$ for every $s \in S$.
\end{itemize}

Furthermore, we say that a family of pairs is \emph{well defined} if it also satisfies the following property:

\begin{itemize}
    \item $mD$ is Cartier locally around the generic point of each irreducible component of $X_s \cap \Supp(D)$ for every $s \in S$, where $m \in \nn$ is a sufficiently divisible natural number clearing the denominators of $D$.
\end{itemize}

This latter condition guarantees that $mD$ is Cartier on a big open set $U \subset X$ with the property that $U \cap X_s$ is a big open set of $X_s$ for every $s \in S$.
This guarantees that we have a well-defined notion of pull-back of $D$ under any possible base change $S' \rar S$, as we can pull back $mD|_U$, take its closure in $X \times_S S'$, and then divide the coefficients by $m$.

There is a more general definition of families of divisors, over possibly non-reduced bases due to Koll\'ar in \cite{kol19s}. We will not report it here since we will not need it, we refer the interested reader to \emph{loc. cit.}

A well defined family of pairs $f \colon (X,D) \rar S$ over a reduced base $S$ is called \emph{locally stable} if, for every base change $g \colon (X_C,D_C) \rar C$ where $C$ is the spectrum of a DVR with closed point $0$, $(X_C,D_C+X_0)$ is a semi-log canonical pair. 
Then, a family $f \colon (X,D) \rar S$ is called \emph{stable} if it is locally stable, $f$ is proper, and $\K X/S. + D$ is $f$-ample.

\subsection{Families of $\bQ$-pairs}

Let us fix a positive integer $n$, a polynomial $p(t)\in \bQ[t]$, and a finite set of coefficients $I \subset (0,1] \cap \qq$.
Let $r$ be the index of $I$, see \S~\ref{indice}.
Before we introduce our main functor, we discuss some applications of the abundance conjecture that are relevant for this paper.
\begin{Notation}

Throughout the paper, when we say ``assume that assumption (A) holds'' we mean ``assume that the following condition holds'': 
\begin{center}
if $(X,D)$ is a log canonical pair such that $K_X + (1-\epsilon)D$ is ample for $0<\epsilon \ll 1$, then $K_X + D$ is semi-ample.
\end{center}
\end{Notation}
\begin{Remark}
We remark that assumption (A) is a weakening of the Abundance Conjecture, which is known up to dimension 3.
Moreover, assumption (A) is known to hold for klt pairs $(X,D)$ (it is the basepoint-free theorem).

\end{Remark}

\begin{Def}\label{Def:functor}
Let $\sF_{n,p,I}'$ be the category fibered in groupoids over $\operatorname{Sch}/\bC$ whose fibers over a scheme $B$  consists of:
\begin{itemize}
    \item a flat and proper morphism $f \colon \cX\to B$ of relative dimension $n$;
    \item a flat and proper morphism $\cD\to B$ of relative dimension $n-1$ and relatively $S_1$;
    \item a closed embedding $i \colon \cD\to \cX$ over $B$; and
    \item for every point $b\in B$, the fiber $(\cX_b;\frac{1}{r}\cD_b) $ is $\bQ$-stable with polynomial $p(t)$ and coefficients in $I$, where $r$ is the index of $I$, see \S~\ref{indice}.
\end{itemize}
For every morphism $B'\to B$, we denote by  $j_{B'} \colon \cX\times_B B'\to \cX$ the first projection, and by $\cX'$ the fiber product $\cX\times_B B'$.
Similarly, for every point $b\in B$, we denote by $\cD_b \coloneqq \spec(k(b))\times_B\cD$.
If we denote by $\cI_\cD$ the ideal sheaf of $\cD$, we require that for every $B'\to B$ and every $m$, $m'$, the natural map \begin{equation}\label{eq_K_cond}\tag{K}
    j_{B'}^*((\omega_{\cX/B}^{\otimes m'}\otimes\cI_D^{\otimes m})^{[1]})\to (\omega_{\cX'/B'}^{\otimes m'}\otimes j^*_{B'}(\cI_D^{\otimes m}))^{[1]}\end{equation} is an isomorphism.
We will denote by $(\cX;\cD)\to B$ an object of $\sF_{n,p,I}'$ over $B$, and we will call it a \emph{weak family of $\bQ$-stable pairs}.
\end{Def}

\begin{Def}
We will call \emph{weak $\bQ$-stable morphism} the datum of a flat and proper morphisms $f \colon \cX\to B$ and a closed embedding $i \colon \cD\to \cX$ that satisfies the four bullet points of Definition \ref{Def:functor}.
If there is no ambiguity we will still denote it with $(\cX;\cD)\to B$.
Lastly, if $\K \mathcal{X}/B. + \frac{1}{r}\mathcal{D}$ is relatively semi-ample, we say it is a {\it $\bQ$-stable morphism}.
\end{Def}

The condition (K) on commutativity with base change in Definition \ref{Def:functor} is usually referred to as \emph{Koll\'ar's condition}. 

We need an additional condition to prove that $\sF_{n,p,I}'$ is representable in full generality. Indeed, while proving representability if one restricts itself to the category of \emph{reduced} schemes (i.e., if one is only interested in families over a reduced base) essentially follows from \cite{HX13}, to prove that $\sF_{n,p,I}'$ is representable in general we will need some version of assumption (A) to hold in families. If one works with moduli of surfaces or threefolds, then the Abundance Conjecture is known and there are no issues.
Otherwise one possible solution, which was suggested to us by Koll\'ar, is to only consider families $(\cX;\cD)\to B$ such that ``assumption (A) holds in families''.
The advantage of this approach is that the functor can be defined in any dimension unconditionally to the Abundance Conjecture.

\begin{Def}[Koll\'ar]\label{def families A+} Let $\sF_{n,p,I}$ be the functor representing families $(\cX;\cD)\to B$ in $\sF_{n,p,I}'$ such that
there is a family of stable pairs $(\cX^s,\frac{1}{r}\cD^s)\to B$ together with a morphism $\cX\to \cX^s$ such that for every $b\in B$, the restriction $(\cX_b,\frac{1}{r}\cD_b)\to (\cX^s_b;\frac{1}{r}\cD^s_b)$ is the canonical model of $(\cX_b,\frac{1}{r}\cD_b)$. We will denote these families as \textit{families of $\bQ$-stable pairs.}
\end{Def}
The main advantage of dealing with assumption (A) in this way
is proved \cite{kol_new}*{Proposition 8.36}, which he  kindly shared with us, where he proves that imposing assumption (A) in families is a constructible condition:

\begin{theorem}[Koll\'ar]\label{teorema kollar constructible}
Given a proper locally stable morphism
$(X,\Delta)\to S$ there is a locally closed partial decomposition $S'\to S$ such that for any $T\to S$, the pull-back $(X_T,\Delta_T)\to S$ has a simultaneous, canonical, crepant, birational contraction $(X_T,\Delta_T)\to(X_T^s,\Delta_T^s)\to T$ with $(X_T^s,\Delta_T^s)\to T$ stable iff $T\to S$ factors via $T\to S'$.
\end{theorem}
We will use this result for constructing $\sF_{n,p,I}$, and to show that it is bounded.

\newtheorem*{Summary_notation*}{Summary of notations}


\begin{Summary_notation*}
As it might be confusing to remember all the different definitions of families, we recall the essential differences here:
\begin{enumerate}
    \item in \textit{weak family of $\bQ$-stable pairs}, we require condition (K) but not condition (A);
    \item in \textit{weak} $\bQ$\textit{-stable morphism}, we require neither condition (K) nor condition (A);
    \item in $\bQ$\textit{-stable morphism}, we require condition (A) but not condition (K); and
    \item in \textit{family of} $\bQ$\textit{-stable pairs}, we require both condition (A) and condition (K).
\end{enumerate}
In particular, we antepone ``weak'' if we are not requiring condition (A), and we write ``family'' if we require condition (K).
This choice will be maintained when introducing the notion of ``constant part'' in \S~\ref{sec_constant_part}.
\end{Summary_notation*}

Now, we add a series of remarks and technical statements that are relevant in this context.
We keep the notation of Definition \ref{Def:functor}.
\begin{Remark}
Observe that, with the notation of Definition \ref{Def:functor} and Notation \ref{notation_Dsc}, $(\frac{1}{r}\cD_b)^{sc} = \cD_b$.
\end{Remark}
\begin{Remark} \label{remark_ideal_flat}
The ideal sheaf $\mathcal{I}_\cD$ is flat over $B$.
This follows from the fact that $\O \cX.$ and $\O \cD.$ are flat, by considering the associated long exact sequence of $\operatorname{Tor}$.
\end{Remark}

\begin{Remark}
Since $f$ is flat and by condition (3) its fibers are $S_2$, $f$ is relatively $S_2$.
Then, since $\cD\to B$ is flat and relatively $S_1$, it follows from \cite{Kol13}*{Corollary 2.61} that $\cI_\cD$ is relatively $S_2$.
\end{Remark}

\begin{Remark}
By condition (3), there exists an open subset $V\subseteq \cX$ whose restriction to any fiber is a big open subset, such that $\cD_b$ is a Cartier divisor on $V_b$ for every $b \in B$.
Then, by flatness, we may apply \cite{stacks-project}*{Tag 062Y}, and we conclude that $\cD$ is Cartier along $V$.
Then, by Remark \ref{remark_ideal_flat} and \cite{HK04}*{Proposition 3.5}, $\O \cX.$ and $\cI_\cD$ are reflexive.
In particular, we have $(\cI_\cD)^{[1]} = \cI_\cD$, and $\cI_\cD$ satisfies the conditions in \cite{kol_new}*{Definition 3.28}.
\end{Remark}

\begin{Remark}
Observe that in Definition \ref{Def:functor}, the divisor $\cD_b$ in point (3) was defined via a fiber product.
However, according to our conventions (see \S~\ref{divisorial sheaves}) it should also correspond to a Weil divisor on $\cX_b$.
This is true since $\cD\to B$ is relatively $S_1$, so $\cD_b$ is $S_1$ so its ideal sheaf in $\cX_b$ is $S_2$ from \cite{Kol13}*{Corollary 2.61}. 
\end{Remark}

\begin{Remark}
The sheaves $\cI_\cD^{[m]}$ are ideal sheaves of $\cO_\cX$.
Indeed, we can again consider an open subset $V\subseteq \cX$ whose restriction to any fiber is a big open subset and such that $\cD$ is a Cartier divisor on $V$.
Then, if we denote by $i \colon V\to \cX$ the inclusion of $V$, by \cite{HK04}*{Corollary 3.7} we have
\[
\cI_\cD^{[m]} = i_*(\cO_V(-m\cD|_{V})).
\]
Then, the inclusion 
$\cO_V(-m\cD|_{V})\hookrightarrow \cO_V$ can be pushed forward via $i$, to have an inclusion $\cI_\cD^{[m]}\hookrightarrow i_*(\cO_V)=\cO_\cX$, where the last equality follows from \cite{HK04}*{Proposition 3.5} since $\cX\to B$ is $S_2$.
\end{Remark}
\begin{Notation}\label{Notation_mD_for_pstable_fam}
If $f \colon (\cX;\cD)\to B$ is a weak family of $\bQ$-stable pairs, we denote by $m\cD$ the closed subscheme of $\cX$ with ideal sheaf $\cI^{[m]}_\cD$. 
\end{Notation}

\begin{Remark}\label{remark:bc:implies:flat}
By \cite{AH11}*{Proposition 5.1.4}, the sheaves $\cI_\cD^{[m]}$ and $\omega_{\cX/B}^{[m]}$ are flat over $B$ for every $m\in \bZ$.
\end{Remark}

\begin{Lemma}
The morphism $m\cD\to B$ is flat with $S_1$ fibers (i.e., the fibers have no embedded points).\end{Lemma}
\begin{proof}
To check that $m\cD\to B$ is flat it suffices to check that for every closed point $\spec(k(b))\to B$ we have $\operatorname{Tor}^1(k(b),\cO_{m\cD}) = 0$. We pull back the exact sequence
\[
0\to \cI_{\cD}^{[m]}\to \cO_{\cX} \to \cO_{m\cD}\to 0
\]
via $j_{b} \colon \cX_b\to \cX$, and we obtain
\[
0 = \operatorname{Tor}^1(k(b),\cO_{\cX}) \to \operatorname{Tor}^1(k(b),\cO_{m\cD}) \to j_{b}^*\cI_{\cD}^{[m]}\to \cO_{\cX_{b}} \to \cO_{(m\cD)_{b}}\to 0.
\]
However, $j_{b}^*\cI_{\cD}^{[m]} \cong (j_{b}^*\cI_{\cD})^{[m]}$, so in particular it is a torsion free sheaf of rank 1.
Then, the map $j_{b}^*\cI_{\cD}^{[m]}\to \cO_{\cX_{b}}$ is injective, so $\operatorname{Tor}^1(k(b),\cO_{m\cD})=0$ as desired.

Finally, $\cI_{\cD_p}^{[m]}$ is $S_2$ and $\cO_{\cX_p}$ is $S_2$ from the commutativity with base change (K) in Definition \ref{Def:functor}, so $\cO_{(m\cD)_p}$ is $S_1$ from \cite{Kol13}*{Corollary 2.61}.\end{proof}

\begin{Remark}\label{remark_for_m_big_enough_kollar_comndition_guarantees_cartier}
There is an $m>0$, which does not depend on $B$, such that $\omega_{\cX/B}^{[m]}$ and $\cI_\cD^{[m]}$ are locally free on $\cX$.
Indeed, we will prove in Theorem \ref{thm_boundedness} that there is an $m$ such that for every fiber $(X;D)$ of $f$, the sheaves $\omega_X^{[m]}$ and $\cO_X(-mD)$ are Cartier.
Since our family is bounded (see Theorem \ref{thm_boundedness}), we can choose such an $m$ that does not depend on the basis $B$.
Then, by condition (K) in Definition \ref{Def:functor} and Remark \ref{remark:bc:implies:flat}, we may apply  \cite{stacks-project}*{Tag 00MH}, which implies that $\omega_{\cX/B}^{[m]}$ and $\cI_\cD^{[m]}$ are locally free since they restrict to locally free sheaves along each fiber.
\end{Remark}

Now, we specify the morphisms in the fibered category $\sF_{n,p,I}$ over a morphism $f \colon T\to B.$
Let $\alpha = ((\cY,\cD_\cY)\to T)$ be an element of $\sF_{n,p,I}(T)$, and let $\beta = ((\cX;\cD_\cX)\to B)$ be an element of $\sF_{n,p,I}(B)$.
An arrow $\alpha \to \beta$ is the datum of two morphisms $(g,h)$ that fit in a diagram like the one below, where all the squares are fibered diagrams:
\[
\xymatrix{\cD_\cY\ar[r]^g\ar[d] & \cD\ar[d] \\ \cY\ar[d] \ar[r]^h & \cX\ar[d] \\ T\ar[r]^f & B.}
\]

\begin{Oss}
The only morphisms over the identity $\Id \colon B\to B$ are isomorphisms.
Thus, $\sF_{n,p,I}$ is fibered in groupoids.
\end{Oss}

\subsection{Hassett's example} \label{hassett example}
In this subsection, we present a well-known example due to Hassett that is helpful to keep in mind to navigate the rest
of the paper.
See also \cite{KP17}*{\S~1.2} or \cite{kol_new}. 

Consider the DVR $R = \spec(k[t]_{(t)})$, let $\eta$ (resp. $p$) be the generic (resp. closed) point of $\spec(R)$, and let $\cX=\bP^1\times \bP^1 \times \spec(R)$.
Consider $C$ a smooth member of $\cO_{\bP^1\times \bP^1}(1,2)$, and let $\cY$ be the blow-up of $C$ in the special fiber of $\cX$.
Then, if we compose the blow-down $\cY\to \cX$ with the projection $\cX\to \spec(R)$, we get a family of surfaces $\cY\to \spec(R)$ where the generic fiber is a copy of $\bP^1\times \bP^1$, while the special fiber is a surface with two irreducible components.
One irreducible component of the special fiber is isomorphic to $\bP^1\times \bP^1$
(the proper transform of the special fiber of $\cX\to \spec(R)$), and the other one is the exceptional divisor $F$.
The surface $F$ is the projectivization of the normal bundle of $C$.
Since $C\cong \bP^1$ and the normal bundle of $C$ in $\cX$ is isomorphic to $\cO_{\bP^1}(0) \oplus \cO_{\bP^1}(4)$, we have that $F$ is isomorphic to the Hirzebruch surface $\bF_4$.
We denote by $\Delta\subseteq \bF_4$ the preimage of the double locus of the central fiber on $\bF_4$.

We consider a divisor on $\bP^1\times \bP^1$ consisting of five irreducible components, three general members of $\cO_{\bP^1\times \bP^1}(1,2)$ (which we denote by $C_1, C_2, C_3$), $C$, and a smooth member $G$ of $\cO_{\bP^1\times \bP^1}(2,0)$.
We consider a deformation of $C_1+C_2+C_3+C + G$ in $\cX$ given by the trivial deformation of $C$ and $G$, and we deform $C_i$ to $C$ for every $i$.
We denote by $\cD_\cX$ the total space of this deformation, and by $\cD$ its proper transform in $\cY$.
More explicitly, if $C$ is the zero locus of a global section $h\in H^0(\cO_{\bP^1\times \bP^1}(C))$, $G$ is the zero locus of a global section $g\in H^0(\cO_{\bP^1\times \bP^1}(G))$, and $\varphi_1, \varphi_2, \varphi_3$ are generic sections of $H^0(\cO_{\bP^1\times \bP^1}(C))$, the deformation we consider is $\cD_\cX = V(h(t\varphi_1 - h)(t\varphi_2 - h)(t\varphi_3 - h)g)$.

\begin{center}
\includegraphics[scale=0.19]{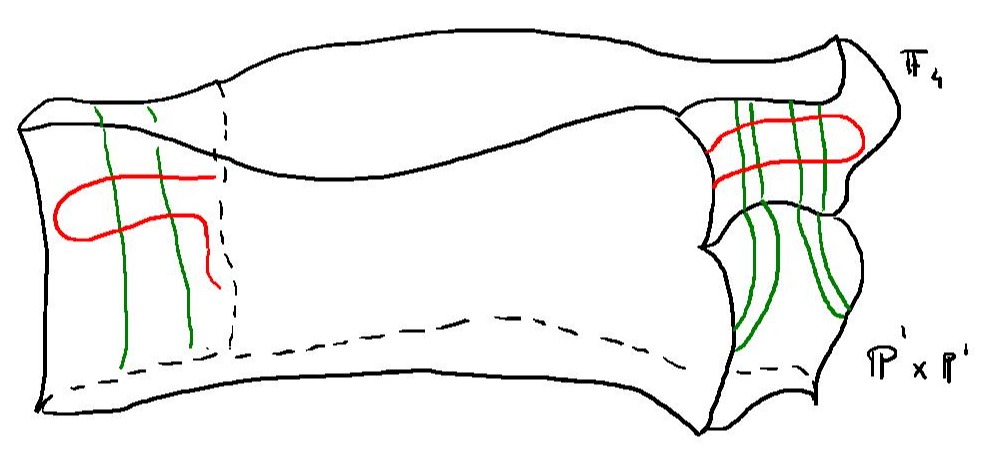}\\ Figure 1: the family $(\cY,c\cD)\to \spec(R).$
\end{center}

Now, we construct the canonical model of $(\cY,c\cD)\to \spec(R)$, with the coefficient $c$ in a neighbourhood of $\frac{1}{2}$.
First, we introduce some notation.
We denote $K_{\cY/\spec(R)}+c\cD$ by $\cL(c)$, the irreducible component of the central fiber isomorphic to $\bP^1\times \bP^1$ by $\sP$, and the two rulings of $\sP$ by $f_1$ and $f_2$, with the convention that $C \equiv f_1+2f_2$.
Since the generic fiber $(\cY,c\cD)_\eta$ is stable for every $c$ with $|c-\frac{1}{2}| \ll 1$, we just need to control the intersection pairings on the special fiber.

Let $c=\frac{1}{2} + \epsilon$, where $|\epsilon| \ll 1$.
\begin{center}
\includegraphics[scale=0.19]{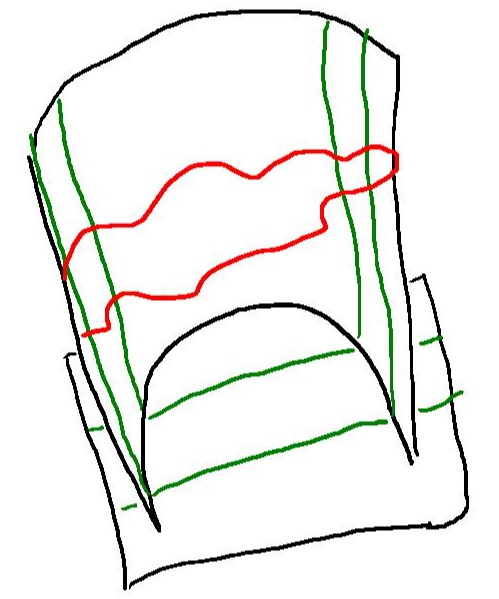} \\ Figure 2: special fiber of the family $(\cY,c\cD)\to \spec(R)$.
\end{center}
One can check that:
\begin{enumerate}
\item when $\epsilon > 0$, the divisor $\cL(\frac{1}{2} + \epsilon)$ is nef.
It is ample on $F$, positive on $f_2$ and 0 on $f_1$;
\item when $\epsilon = 0$, the divisor $\cL(\frac{1}{2})$ is nef.
It is 0 on $\sP$ and on $\Delta$; and
\item when $\epsilon < 0$, the divisor $\cL(\frac{1}{2}+\epsilon)$ is  not nef.
On $\sP$, it is negative on $f_2$ and 0 on $f_1$, while on $F $ it is negative on $\Delta$.
\end{enumerate}
Therefore, we can explicitly describe the special fiber of the canonical model of $(\cY,c\cD)$ over $\spec(R)$. We denote by $(\cZ^+,c\cD^+)$ (resp. $(\cZ^0, c\cD^0)$, $(\cZ^-,c\cD^-)$) the canonical model of $(\cY,c\cD)\to \spec(R)$ when $0<\epsilon\ll 1$ (resp. $\epsilon=0$, $-1 \ll \epsilon < 0$):
\begin{enumerate}
\item when $\epsilon > 0$, to construct the canonical model we contract the ruling $f_1$.
Since via this contraction the map $C\subseteq \bP^1\times \bP^1 \to \bP^1$ is a $2 \colon 1$ ramified cover of $\bP^1$, the special fiber is the push-out of the following diagram:
$$\xymatrix{\bP^1\cong \Delta \ar[r] \ar[d]^{2 \colon 1} \ar[r] & F\ar[d]_\pi \\ \bP^1\ar[r] & \cZ^+_p}$$
\item when $\epsilon = 0$, to construct the canonical model, we contract $\sP$ and $\Delta$.
The special fiber is isomorphic to $\bF_4$ with the section $\Delta$ contracted, which is the projectivization of the cone over a rational quartic curve; and
\item When $\epsilon < 0$, we perform a divisorial contraction to make the divisor nef.
We contract the ruling $f_2$, and the special fiber is $\bF_4$.
\end{enumerate}
In particular, there are morphisms $\pi^+ \colon \bF_4\to \cZ^+_p$,  $\pi^0 \colon \bF_4\to \cZ^0_p$ and $\pi^- \colon \bF_4\to \cZ^-_p$ from $\bF_4$ to the special fibers of $\cZ^+$, $\cZ^0$ and $\cZ^-$.
The divisors $\cD^+_p$, $\cD^0_p$ and $\cD^-_p$ can be described as follows.
First, recall that $\bF_4$ is the projectivization of the normal bundle of $\bP^1$ inside $\cX$, so it has a relative $\cO_{\bF_4}(1)$.
We denote a generic section of $\cO_{\bF_4}(1)$ by $h$.
Then, the following holds:

\begin{enumerate}
\item when $\epsilon > 0$, the divisor $\cD^+_p$ is the image via $\pi^+$ of four general fibers of $F$, together with a divisor linearly equivalent to $4h$.
All the components have coefficient $\frac{1}{2} + \epsilon$;
\item when $\epsilon = 0$, the divisor $\cD^0_p$ is the image via $\pi^0$ of four general fibers with a divisor linearly equivalent to $4h$.
All the components have coefficient $\frac{1}{2}$; and
\item when $\epsilon < 0$, the divisor $\cD^-_p$ consists of four general fibers with a divisor linearly equivalent to $4h$, four generic fibers, and $\Delta$.
All the components have coefficient $\frac{1}{2} + \epsilon$, with the exception of $\Delta$, which has coefficient $1+2\epsilon$.
\end{enumerate}

Recall now that the flat limit of $\cD_{\eta}\subseteq \cZ^0_\eta$ in $\cZ^0_p$ is \emph{not} $S_1$, since it has an embedded point (see \cite{KP17}*{\S~1.2}).
However, $(\cZ^-;\cD^-)\to \spec(R)$ is a $\bQ$-stable morphism with coefficients $I = \{\frac{1}{2}, 1\}$ (see Proposition \ref{prop:flat:limit:is:S1}), so in particular the flat limit of $\cD_\eta$ in $\cZ_0^-$ does \emph{not} have an embedded point on the special fiber.

\subsection{$\bQ$-stable morphisms with constant part}\label{sec_constant_part}
For proving that $\sF_{n,p,I}$ is bounded, it will be useful to introduce the following definition.

\begin{Def}
Assume $S$ is reduced.
A \emph{locally weak $\bQ$-stable morphism with constant part} $f \colon (\cX,\Omega;\cD) \rar S$ over $S$ and with coefficients in $I$ is the datum of a proper morphism $f \colon \cX \rar S$, a closed subscheme $\cD$ on $\cX$, and an effective $\qq$-divisor $\Omega$ on $X$, such that the following conditions hold:
\begin{itemize}
    \item $f \colon (X,\Omega) \rar S$ is a proper, locally stable family of pairs, where $\Omega$ is a family of Mumford divisors \cite{kol19s};
    \item $\cD \rar S$ is flat and relatively $S_1$ (namely, with no embedded points) with fibers of pure dimension $n-1$; and
    \item for every $s \in S$, there is a $\bQ$-pair $(X,\Delta;D)$  with coefficients in $I$ (i.e., both $\Delta$ and $D$ have coefficients in $I$) such that $X=\cX_s$, $\Delta=\Omega_s$, and $D^{sc}=\cD_s$.
\end{itemize}
Furthermore, if we have a polynomial $p(t) \in \qq[t]$ and for every $s \in S$, $(\cX_s,\Omega_s;\frac{1}{r}\cD_s)$ has polynomial $p(t)$, we say that it is a
locally weak $\bQ$-stable morphism with constant part {\it with polynomial $p(t)$}.
Lastly, if $\K \mathcal{X}/S. + \frac{1}{r}\mathcal{D}+\Omega$ is relatively semi-ample, we drop the ``weak'' from the notation.
\end{Def}
We remark that we call $\Omega$ the constant part since, contrary to $\cD$, the divisor $\Omega$ might not be $\bQ$-Cartier on $X$.

\begin{Notation}
We say that a locally (weak) $\bQ$-stable morphism with constant part (with coefficients in $I$ and polynomial $p(t)$) is a \emph{(weak) $\bQ$-stable morphism with constant part} (\emph{with coefficients in $I$ and polynomial $p(t)$}) if, for every $s \in S$, the fiber over $s$ is a $\bQ$-stable pair.
\end{Notation}

\begin{Prop} \label{prop_well_defined_p-pairs}
Let us fix a set of coefficients $I \subset (0,1] \cap \qq$, a polynomial $p$ and an integer $n$.
Let $r$ denote the index of $I$.
Let $f \colon (\cX,\Omega;\cD) \rar S$ be a weak $\bQ$-stable morphism with constant part $\Omega$ over a reduced base $S$.
Then $f \colon (\cX, \frac{1}{r}\cD+\Omega) \rar S$ is a well defined family of pairs.
\end{Prop}
\begin{proof}
By definition $\cD \rar S$ is flat with fibers of pure dimension $n-1$.
Thus, it follows that the fibers of $\mathrm{Supp}(\cD + \Omega) \rar S$ are either empty or of pure dimension $n-1$.
Thus, to show that $f \colon (\cX, \frac{1}{r}\cD + \Omega) \rar S$ is a family of pairs, we are left with showing that $f$ is smooth at the generic points of $X_s \cap \supp(\cD + \Omega)$ for every $s \in S$.

By assumption, this is the case for all the generic points of $X_s \cap \supp(\cD + \Omega)$ arising from $\Omega$.
Thus, we may focus on the contribution of $\cD$.
But then, since each fiber $(\cX_s,\Omega_s;\cD_s)$ is a $\bQ$-pair, it follows that the generic points of $\cD_s$ are contained in the smooth locus of $X_s$ by the semi-log canonical condition.
Thus, $f \colon (\cX, \frac{1}{r}\cD + \Omega) \rar S$ is a family of pairs.

To conclude, we need to show that the family is well defined.
To this end, it suffices we focus on $\cD$, as $\Omega$ satisfied the needed conditions by definition.
As argued in \cite{kol_new}*{Definition 3.35}, there is a big open subset $U \subset \cX$ such that every point of $U$ is either smooth or nodal, and $U \cap \cX_s$ has codimension at least 2 in $\cX_s$ for every $s \in S$.
For every $s \in S$, $\cD_s$ does not contain any irreducible component of the double locus.
Thus, the intersection between $\cD$ and $U$ has codimension at least 2 in every fiber.
Let $V$ denote the open set obtained by removing this intersection from $U$.
Then, as the scheme theoretic restriction $\cD_s$ is an integral Weil divisor, it is a Carter divisor along $V_s$.
Then, by \cite{stacks-project}*{Tag 062Y}, $\cD$ is a Cartier divisor along $V$.
Thus, the claim follows.
\end{proof}

\begin{Prop} \label{prop_stable_family_p-pairs}
Fix an integer $n\in \mathbb{N}$, a polynomial $p(t)\in \mathbb{Q}[t]$, and a finite subset $I\subseteq (0,1] \cap \mathbb{Q}$.
Let $\epsilon_0>0$ be as in Lemma \ref{lemma_bound_for_epsilon}, and let $f \colon (\cX,\Omega;\cD) \rar S$ be a weak $\bQ$-stable morphism with constant part, coefficients in $I$, and polynomial $p(t)$, over a reduced base $S$.
Then $f \colon (\cX, \frac{1-\epsilon}{r}\cD + \Omega) \rar S$ is a stable family of pairs for every $0 < \epsilon < \epsilon_0$ rational.
In particular, $\cD$ is $\qq$-Cartier.
\end{Prop}

\begin{proof}
Fix $\epsilon$ as in the statement.
Then, by Proposition \ref{prop_well_defined_p-pairs}, $f \colon (\cX, \frac{1-\epsilon}{r}\cD + \Omega) \rar S$ is a well defined family of pairs.
By assumption, the self-intersection $(\K \cX_s. + \frac{1-\epsilon}{r}\cD_s + \Omega_s)^n$ is independent of $s \in S$, as it is $p(1-\epsilon)$.
Since $f \colon (\cX, \frac{1-\epsilon}{r}\cD + \Omega) \rar S$ is a well defined family of pairs, we may find a big open subset $U \subset \cX$ such that $U_s$ has codimension at least 2 for every $s \in S$ and $\K \cX/S. + \frac{1-\epsilon}{r}\cD + \Omega$ is $\qq$-Cartier along $U$.
Then, $\K \cX/S. + \frac{1-\epsilon}{r}\cD + \Omega$ is $\qq$-Cartier by \cite{kol_new}*{Theorem 5.8}.
The claim follows by \cite{kol_new}*{Definition-Theorem 4.7} and the fact that the argument was independent of $\epsilon \in (0,\epsilon_0) \cap \qq$.
\end{proof}

\subsection{Existence of good minimal models}
Let $(X,\Delta)$ be a log canonical pair, and let $f \colon X \rar S$ be a projective morphism over a normal variety $S$ such that $\K X. + \Delta$ is $f$-pseudo-effective.
Then, it is expected that $X$ admits a good minimal model over $S$.
That is, $X$ admits a birational contraction $\phi \colon (X,\Delta) \drar (X',\Delta')$ over $S$ to a log canonical pair $(X',\Delta')$, such that $\Delta'$ is the push-forward of $\Delta$ to $X'$, $\phi$ is $(K_X+\Delta)$-negative, and $\K X'. + \Delta'$ is semi-ample over $S$.

Here, we collect a technical statement that shows the existence of relative good minimal models under certain assumptions.
In particular, this statement is crucial to show that, under suitable hypotheses, a weak $\bQ$-stable morphism (resp. weak family of $\bQ$-stable pairs) is actually a $\bQ$-stable morphism (resp. family of $\bQ$-stable pairs).

\begin{Lemma} \label{lemma minimal models}
Let $(X,\Delta)$ be a log canonical pair, and let $f \colon X \rar S$ be a projective morphism to a normal variety such that $\K X. + \Delta$ is $f$-pseudo-effective.
Assume that the general fiber of $f$ has a good minimal model.
Then,
$(X,\Delta)$ admits a relative good minimal model over $S$.
Furthermore, if $\K X. + \Delta$ is nef over $S$, then it is semi-ample.
\end{Lemma}

\begin{proof}
Let $\pi \colon X' \rar X$ be a log resolution of $(X,\Delta)$, and let $\pi^*(K_X+\Delta)= \K X'. + \Delta' + E' + \Gamma' - F'$.
Here $\Delta'$ denotes the strict transform of $\Delta$, the divisors $E'$, $\Gamma'$, and $F'$ are effective, $\pi$-exceptional, and share no common components.
Furthermore, $\Gamma'$ is reduced, while the coefficients of $E'$ are in $(0,1)$.
Let $\Xi'$ denote the reduced $\pi$-exceptional divisor, and fix a rational number $0 < \epsilon \ll 1$.
Then, $(X',\Delta' + E' + \Gamma' + \epsilon (\Xi' - \Gamma'))$ is dlt, and it has the same pluricanonical ring as $(X,\Delta)$.
Furthermore, by the addition of $\epsilon (\Xi' - \Gamma')$, every $\pi$-exceptional divisor that is not in $\Gamma$ is in the relative stable base locus of $\K X'. +\Delta' + E' + \Gamma' + \epsilon (\Xi' - \Gamma')$.
Finally, by assumption and the choice of $\epsilon$, every log canonical center of $(X',\Delta' + E' + \Gamma' + \epsilon (\Xi' - \Gamma'))$ dominates $S$.

By assumption, the general fiber has a good minimal model.
Thus, by \cite{HMX18}*{Theorem 1.9.1}, it follows that $(X',\Delta' + E' + \Gamma' + \epsilon (\Xi' - \Gamma'))$ has a relative good minimal model over a non-empty smooth affine open subset $U \subseteq S$.
Then, as by assumption there are no vertical log canonical centers, it follows from \cite{HX13}*{Theorem 1.1} that $(X',\Delta' + E' + \Gamma' + \epsilon (\Xi' - \Gamma'))$ has a relative good minimal model over $S$.
Since every $\pi$-exceptional divisor that is not in $\Gamma'$ is in the relative stable base locus, any such divisor is contracted on the minimal model.
This shows that the achieved model is a relative good minimal model of $(X,\Delta)$ over $S$ in the sense of Birkar--Shokurov (see \cite{LT22}*{Definition 2.8}); that it, it is a good minimal model where we allow to extract some log canonical places.
But then, by \cite{LT22}*{Lemma 2.9}, the existence of such model also implies the existence of a minimal model in the usual sense.
In turn, this latter model is also good by \cite{HMX}*{Lemma 2.9.1}, and the first part of the claim follows.

Now, assume that $K_X + \Delta$ is relatively nef.
Then, $(X,\Delta)$ is a relative weak log canonical model for $(X',\Delta' + E' + \Gamma' + \epsilon (\Xi' - \Gamma'))$ in the sense of \cite{HMX}.
Then, we conclude by \cite{HMX}*{Lemma 2.9.1} that $(X,\Delta)$ is a relatively semi-ample model.

In the course of the proof, we used \cite{HMX}*{Lemma 2.9.1} in the relative setting.
Notice that \cite{HMX}*{Lemma 2.9.1} is phrased for projective pairs.
On the other hand, by \cite{HX13}*{Corollary 1.2}, one can first take a projective closure of $\tilde X$ over a compactification $\overline{S}$ of $S$, and take a projective relative good minimal model.
Then, by adding the pull-back of some divisor on $\overline{S}$, we can regard the relative good minimal model as a projective minimal model.
Thus, it follows that $K_{X} + \Delta$ is relatively semi-ample over $S$, and it defines a morphism to the relative canonical model.
\end{proof}

\section{Boundedness}\label{section_boundedness}
The goal of this section is to prove that, if we fix a set of coefficients $I$, a polynomial $p$ and a dimension $n$, the corresponding set of $\bQ$-stable pairs is effectively log-bounded.

\begin{Prop}\label{prop:nef:open}
Fix an integer $n\in \mathbb{N}$.
Consider a locally stable family of relative dimension $n$ $(\cX,\cD)\to B$ over a reduced scheme $B$, with $\cD$ being $\bQ$-Cartier.
Assume that there is an $0<\epsilon_0 < 1$ such that  $(\cX,(1-\epsilon_0)\cD)\to B$ is stable.
Then, the set of points $b\in B$ such that $K_{\cX_b} + \cD_b$ is semi-ample is constructible.
\end{Prop}

\begin{proof}
This follows immediately from Theorem \ref{teorema kollar constructible}.
\end{proof}

\begin{theorem}\label{thm_boundedness}
Fix an integer $n\in \mathbb{N}$, a finite subset $I\subseteq (0,1] \cap \mathbb{Q}$, and a polynomial $p(t) \in \mathbb{Q}[t]$.
Then, the set of $\bQ$-stable pairs $(X,\Delta;D)$ of dimension $n$, polynomial $p(t)$ and coefficients in $I$ is effectively log bounded.
\end{theorem}
\begin{proof}
We proceed in several steps.

{\bf Step 1:} In this step, we show that the $\bQ$-stable pairs of interest are log bounded.

From Lemma \ref{lemma_bound_for_epsilon}, there is an $\epsilon_0>0$ such that, for every $\bQ$-stable pair $(X,\Delta;D)$ as in the statement, $(X,(1-\epsilon_0)D + \Delta)$ is a stable pair.
Then, from \cite{HMX}, there is a bounding family $(\cX,\mathcal{E})\to B$ of stable pairs of volume $p(1-\epsilon_0)$, coefficients in the finite set $(1-\epsilon_0)I\cup I \cup \{1\}$ and dimension $n$.

{\bf Step 2:} In this step, we show that the $\bQ$-stable of pairs of interest are strongly log bounded.
Furthermore, we may choose the family to be locally stable.

Since the set of coefficients involved is finite, up to taking finitely many copies of the family in order to assign coefficients to $\mathcal{E}$, we may find divisors $\mathcal{D}$ and $\Omega$ supported on $\mathcal{E}$ such that $(1-\epsilon_0)\mathcal{D}$ restricts to $(1-\epsilon_0)D$ fiberwise, and $\Omega$ restricts to $\Delta$ fiberwise.
By \cite{kol_new}*{Lemma 4.44}, up to replacing $B$ with a finite disjoint union of locally closed subsets, we can further assume that both $(\cX,(1-\epsilon)\cD + \Omega)\to B$ and
$(\cX,\cD + \Omega)\to B$ are locally stable.
In particular, $K_{\cX/B} + t\cD + \Omega$ is $\mathbb{Q}$-Cartier for any $t$.
Furthermore, by flatness, up to disregarding some irreducible components of $B$, we can assume that for every $b\in B,$ $p(t) = (K_{\cX/B} + t\cD + \Omega)^{\dim(X)}$.
Finally, up to stratifying $B$, we may assume that each irreducible component of $B$ is smooth;
in particular, $K_B$ is well defined, and it follows that the pairs are strongly log bounded. 

{\bf Step 3:} In this step we finish the proof.

By construction, for the choice of $t = 1-\epsilon_0$, $K_{\cX/B} + (1-\epsilon_0)\cD + \Omega$ is ample on the general fibers.
Thus, up to removing some proper closed subset of $B$, we may assume that $(\cX,(1-\epsilon_0)\cD + \Omega)\to B$ is a stable family.
Thus, to conclude the proof it suffices to use Proposition \ref{prop:nef:open}, which guarantees that the set $
\{p \in B  \colon  (K_{\cX/B} + \cD + \Omega)|_{{\cX}_p}\text{ is semi-ample}\}$ is constructible.
\end{proof}

\section{The moduli functor}
\label{section functor}

The goal of this section is to prove that $\sF_{n,p,I}$ is an algebraic stack. We begin by the following proposition:

\begin{Prop}
The fibered category $\sF_{n,p,I}$ is a stack.
\end{Prop}

\begin{proof}
Since our argument follows the same strategy in \cite{Alp21}*{Proposition 1.4.6}, we only sketch the salient steps here.
The role that in \emph{loc. cit.} is the one of $\omega_\sC^{\otimes 3}$, for us is $L \coloneqq \cO_{X}(m(K_X + (1-\epsilon)D))$, where $\epsilon$ and $m$ are chosen such that $L$ is very ample with $h^i(X,L)=0$ for $i \geq 1$ and such that $H^0(X,L)\to H^0(D^{sc},L|_{D^{sc}})$ is surjective.
These $m$ and $\epsilon$ can be chosen uniformly by Theorem \ref{thm_boundedness}. 

The fact that isomorphisms are a sheaf in the \'etale topology of $\sF_{n,p,I}$ follows from descent as in \cite{Alp21}*{Proposition 1.4.6}.

For proving that $\sF_{n,p,I}$ satisfies descent, we begin by the following observation.
Consider an object $f \colon (\cX,\cD)\to B$ of $\sF_{n,p,I}(B)$, and pick $m$ and $\epsilon$ as before. 
Then, up to replacing $m$ with some uniform multiple, from Remark \ref{remark_for_m_big_enough_kollar_comndition_guarantees_cartier} and from cohomology and base change, we can assume that $\cG \coloneqq \omega_{\cX/B}^{[m]}\otimes \cI_\cD^{[-\frac{m(1-\epsilon)}{r}]}$ is relatively very ample, and $f_*\cG$ is a vector bundle over $B$.
Indeed, by the deformation invariance of $\chi(\cX_b,\cG_b)$ and the vanishing of $H^i(X,L)$ for $i \geq 1$, it follows that the sections of $H^0(\cX_b,\cG_b)$ are deformation invariant, and hence $f_*\cG$ is a vector bundle.
Then, on every affine open trivializing $f_*\cG$, the morphism $\cX \rar \bP(f_*\cG)$ can be identified with $\cX \rar B \times \bP (H^0(\cX_b,\cG_b))$, and the latter is an embedding as it is an embedding over $B$ fiber by fiber, by Nakayama's lemma.
This gives an embedding $\cX\hookrightarrow \bP(f_*\cG)$, and composing it with $\cD\hookrightarrow \cX$ an embedding $\cD\hookrightarrow \bP(f_*\cG)$.
Now the proof is analogous to the one in \cite{Alp21}*{Proposition 1.4.6}.
\end{proof}

\begin{theorem}
Fix an integer $n\in \mathbb{N}$, a finite subset $I\subseteq (0,1] \cap \mathbb{Q}$, and a polynomial $p(t) \in \mathbb{Q}[t]$. Then $\sF_{n,p,I}$ is an algebraic stack.
\end{theorem}

\begin{proof}Let $r$ denote the index of $I$.
We will proceed in several steps.

{\bf Step 1:} In this step, we fix some invariants and consider a suitable Hilbert scheme parametrizing (among others) the total spaces of the $\bQ$-pairs of interest.

By Lemma \ref{lemma_bound_for_epsilon}, we may find a rational number $\epsilon \in (0,1)$ such that $K_X+(1-\epsilon)D$ is ample for every $\bQ$-stable pair $(X;D)$ with polynomial $p(t)$ and coefficients in $I$.
Without loss of generality, we may assume that $\epsilon = \frac{1}{k}$ for a suitable $k \in \nn$.
By Theorem \ref{thm_boundedness}, we may consider
a weak $\bQ$-stable morphism $\pi \colon (\cX;\cD)\to B$
with coefficients in $I$, polynomial $p(t)$, and of relative dimension $n$ that is effectively log bounding for our moduli problem.
By stratification of $B$, we may assume that $B$ is smooth, $\pi$ and $\pi|_{\cD}$ are flat, and that $\pi|_{\cD}$ has $S_1$ fibers.
Furthermore, by Proposition \ref{prop_stable_family_p-pairs}, $\pi$ also induces a family of stable pairs.
In particular, we can also regard $\cD$ as a divisor, not only as a subscheme, and we will be free to take its multiples.
Similarly, any natural multiple $k\cD$ can be regarded as a subscheme, by considering the vanishing locus of $\O \cX.(-k\cD)$.
Then, there is $m>0$ such that, for every $\bQ$-stable pair $(X;D)$ in our moduli problem, we have:
\begin{enumerate}
    \item $mD$ and $m(1-\epsilon)D$ are Cartier;
    \item $mK_X$ is Cartier; and
    \item $m(K_X+(1-\epsilon)D)$ is a very ample line bundle with vanishing higher cohomologies that embeds $X$ into $\bP(H^0(X,\cO_X (m(K_X+(1-\epsilon)D)))$.
\end{enumerate}
Furthermore, up to taking a multiple, we may assume that $\frac{m}{r}$ and $\frac{m(1-\epsilon)}{r}$ are integers, and
\begin{enumerate}
    \item $\frac{m}{r} \cD$ and $\frac{m(1-\epsilon)}{r} \cD$ are Cartier;
    \item $mK_{\cX/B}$ is Cartier; and
    \item $m(K_{\cX/B}+\frac{1-\epsilon}{r}\cD)$ is a relatively very ample line bundle.
\end{enumerate}

By boundedness and upper semi-continuity of the space of global sections, $h^0(X,\cO_X (m(K_X+(1-\epsilon)D))$ attains finitely many values.
Then, we have finitely many polynomials $q_1,\ldots,q_l$ such that the fibers of $\pi$ have Hilbert polynomial $q_i$ for some $i$, for the relatively ample line bundle $\cO_\cX(m(\K \cX/B.+\frac{1-\epsilon}{r}\cD))$.
We consider a union of Hilbert schemes for the polynomials $q_i$, and we denote such a union with $\sH_0$.
Over $\sH_0$, we have a universal family $f_0 \colon \sX\subseteq \sH_0\times \bP^N\to \sH_0$, where the fibers are closed subschemes of $\bP^N$ with Hilbert polynomial $q_i$ for some $i$.
Here, we observe that $N$ may actually attain finitely many distinct values, as we are assuming that each $X$ is embedded with a full linear series; on the other hand, we will work on one Hilbert scheme at the time, thus, by abusing notation, we will simply write $N$.

{\bf Step 2:} In this step, we highlight the strategy for the construction of the moduli functor.

We will construct our moduli functor as a subfunctor of a suitable relative Hilbert scheme, modulo the action of $\mathrm{PGL}_{N+1}$.
For this reason, we will shrink $\sH_0$ to cut the locus of interest for our moduli problem.
In doing so, we have to guarantee that this locus is locally closed and has a well-defined scheme structure.
If we shrink to an open subset, there is no ambiguity in the scheme structure.
On the other hand, if we need to consider a closed or locally closed subset, we need to show this choice has a well-defined scheme structure, which will be functorial in nature.
Finally, we need to guarantee that, at each step, the locus we consider is invariant under the action of $\mathrm{PGL}_{N+1}$.

{\bf Step 3:} In this step, we cut the locus parametrizing demi-normal schemes.

Since being $S_2$ is an open condition for flat and proper families \cite{EGAIV}*{Theorem 12.2.1}, and since small deformations of nodes are either nodes or regular points, up to shrinking $\sH_0$ we may assume that the fibers of $f_0$ are $S_2$ and nodal in codimension one.
That is, the fibers are demi-normal.

{\bf Step 4:} In this step, we cut the locus parametrizing varieties embedded with a full linear series.

Let $(X;D)$ be a $\bQ$-pair in our moduli problem, and let $I_X$ denote its ideal sheaf in $\pr N.$.
Then, by assumption we have $\O {\mathbb P ^N}.(1)|_X \cong \O X. (m(\K X. + (1-\epsilon)D))$ and the higher cohomologies of both sheaves vanish.
Thus, if we consider the short exact sequence
\[
0 \rar I_X \otimes \O {\mathbb P ^N}.(1) \rar \O {\mathbb P ^N}.(1) \rar \O {\mathbb P ^N}.(1)|_X \rar 0,
\]
it provides the following long exact sequence of cohomology groups
\[
0 \rar H^0(\pr N.,I_X \otimes \O {\mathbb P ^N}.(1)) \rar H^0(\pr N.,\O {\mathbb P ^N}.(1)) \rar H^0(X,\O {\mathbb P ^N}.(1)|_X) \rar H^1(\pr N.,I_X \otimes \O {\mathbb P ^N}.(1))=0,
\]
where the vanishing of $H^1(\pr N.,I_X \otimes \O {\mathbb P ^N}.(1))$ follows from the surjectivity of the map $H^0(\pr N.,\O {\mathbb P ^N}.(1)) \rar H^0(X,\O {\mathbb P ^N}.(1)|_X)$ and the vanishing of the higher order cohomologies of $\O {\mathbb P ^N}.(1)|_X$.
By definition of Hilbert scheme, $\sX \rar \sH_0$ is flat, and the subschemes parametrized correspond to a flat quotient sheaf of $\O \sX.$.
Thus, as we have a short exact sequence of sheaves where the last two terms are flat over the base, then so is the first term, which is the family of ideal sheaves of the subschemes of interest.
Thus, upper semi-continuity of the dimension of cohomology groups applies, and we may shrink $\sH _0$ to the open locus parametrizing varieties $Y$ with $H^0(\pr N.,I_Y \otimes \O {\mathbb P ^N}.(1))=0$.
This guarantees that, for every such variety $Y$, any automorphism of $Y$ preserving $\O {\mathbb P ^N}.(1)|_Y$ is induced by an automorphism of $\pr N.$.

{\bf Step 5:} In this step, we introduce a relative Hilbert scheme, in order to parametrize the boundaries of the $\bQ$-pairs of interest.

Proceeding as in Step 1, for every $\bQ$-pair $(X;D)$ in our moduli problem, we may consider the Hilbert polynomial of $rD$ with respect to $\cO_X(m(\K X.+(1-\epsilon)D))$.
Here, recall that $\cD$ corresponds to $rD$ fiberwise, hence the choice of Hilbert polynomial for $rD$ rather than for $D$.
By effective log boundedness and generic flatness of $\cD$, there exist finitely many such polynomials.
As before, we will deal with one Hilbert polynomial at the time, and omit this choice from the notation.

Now, $f$ is projective over $\sH_0$.
In particular, if we pull back the ample line bundle $\cO_{\sH_0\times \bP^N}(1)$ to get a relatively very ample line bundle $\sG$ on $\sX$, we can take the relative Hilbert scheme for the morphism $f$, the line bundle $\sG$, and the polynomial determined by $\cD$ (see \cite{ACH11}*{Ch. IX}).
This gives a scheme $\sH_1\to \sH_0$, together with an universal family $\sD\subseteq \sX_1 \coloneqq \sX\times_{\sH_0} \sH_1$.
Then, as $\sG$ corresponds to $\cO_X(m(\K X.+(1-\epsilon)D))$ on the elements of our moduli problem, every $\bQ$-pair of interest appears as a fiber of this family.

{\bf Step 6:} In this step, we shrink $\sH_1$ to an open subset such that $\sD \cap \mathrm{Sing}(\sX_1 \rar \sH_1)$ has codimension at least 2 along each fiber and such that the ideal sheaf $\sI_\sD$ of $\sD$ is relatively $S_2$.

For every $\bQ$-pair $(X;D)$, $\supp(D)$ does not contain any component of the double locus of $X$.
Thus, by upper semi-continuity of the dimension of the fibers of a morphism, we may shrink $\sH_1$ to an open subset such that $\sD \cap \mathrm{Sing}(\sX_1 \rar \sH_1)$ has codimension at least 2 along each fiber.

Now, by Step 3, all the fibers of $\sX_1 \rar \sH_1$ are demi-normal.
Thus, we may find an open subset $V \subset \sX_1$ such that the following properties hold:
\begin{enumerate}
    \item for every $s \in \sH_1$, $V_s$ is a big open set in $\sX_{1,p}$;
    \item $\sD \cap \mathrm{Sing}(\sX_1 \rar \sH_1) \subset \sX_1 \setminus V$; and
    \item the fibers of $V \rar \sH_1$ have at worst nodal singularities.
\end{enumerate}
Then, $V \rar \sH_1$ is a Gorenstein morphism, so, by \cite{stacks-project}*{Tag 0C08}, $\omega_{\sX_1/\sH_1}$ is an invertible sheaf along $V$.
Furthermore, by \cite{stacks-project}*{Tag 062Y}, the ideal sheaf $\sI_\sD$ is locally free along $V$.
Finally, up to shrinking $\sH_1$, we may assume that $\sI_\sD$ is relatively $S_2$.

Indeed, $\sX_1$ is flat over $\sH_1$, as $\sX$ is flat over $\sH_0$, and $\sD$ is flat over $\sH_1$ by definition of relative Hilbert scheme.
Thus, $\sI_\sD$ is flat over $\sH_1$, as it is the kernel of a surjection of flat sheaves.
Then, we conclude, as being $S_2$ is an open condition for flat and proper families \cite{EGAIV}*{Theorem 12.2.1}.
Notice that, by \cite{HK04}*{Proposition 3.5}, we have $\sI_\sD=\sI_\sD^{[1]}$.

Notice that in this step we shrank $\sH_1$ twice, and both times the process is invariant under the natural action of $\mathrm{PGL}_{N+1}$, as the locus is characterized by properties of the fibers.

{\bf Step 7:} In this step, we cut $\sH_1$ to a locally closed subset $\sH_2$ to ensure that $(\omega_{\sX_1/\sH_1}^{[\mu']}\otimes\sI_\sD^{[\frac{\mu}{r}]})^{[1]}$ is flat and $S_2$ over $\sH_1$, commutes with base change for every $\mu,\mu'\in \bZ$, and that $\sI_\sD^{[\frac{m(1-\epsilon)}{r}]}$, $\sI_\sD^{[\frac{m}{r}]}$, and $\omega_{\sX_1/sH_1}^{[m]}$ are invertible sheaves.
Here, if $\frac{\mu}{r}$ is not an integer, we set $\sI_\sD^{[\frac{\mu}{r}]} \coloneqq  \left(\sI_\sD^{\lfloor\frac{\mu}{r}\rfloor}\right)^{[1]}$.

The first claim follows immediately from \cites{kol08} applied to the sheaves $\sI_\sD^{[\frac{\mu}{r}]}$ and $\omega_{\sX_1/\sH_1}^{[\mu']}$ for $0 \leq \frac{\mu}{r}, \mu' \leq m$. Thus, there is a stratification into functorial locally closed subsets of $\sH_1$, which we denote by $C_i\subseteq \sH_1$, where the above sheaves are flat over $C_i$, they are $S_2$, and they commute with base change. Then, by Remark \ref{remark_for_m_big_enough_kollar_comndition_guarantees_cartier}, the sheaves $\sI_\sD^{[\frac{m(1-\epsilon)}{r}]}$, $\sI_\sD^{[\frac{m}{r}]}$,  and $\omega_{\sX_1/\sH_1}^{[m]}$ are invertible over each component $C_i$.

Considering the union of the $C_i$'s produces another base $\sH_2$, with a family $f_2 \colon \sX_2 \to \sH_2$ and a closed subset $\sD_2 \subset \sX_2$ as the one over $\sH_1$, but such that $\sI_{\sD_2}^{[\frac{m(1-\epsilon)}{r}]}$, $\sI_{\sD_2}^{[\frac{m}{r}]}$, and $\omega_{\sX_2/\sH_2}^{[m]}$ are Cartier and the formation of $(\omega_{\sX_1/\sH_1}^{[\mu']}\otimes\sI_\sD^{[\frac{\mu}{r}]})^{[1]}$ commutes with base change for every $0\leq \frac{\mu}{r}, \mu' \leq m$.
Then, the formation of
$(\omega_{\sX_1/\sH_1}^{[\mu']}\otimes\sI_\sD^{[\frac{\mu}{r}]})^{[1]}$ commutes with base change for every $\mu, \mu'$, since we can write $\mu=km+b$ and $\lfloor\frac{\mu}{r}\rfloor=k'm+b'$ for $0\le b, b' <m$, so we can write $(\omega_{\sX_1/\sH_1}^{[\mu']}\otimes\sI_\sD^{[\frac{\mu}{r}]})^{[1]}$ as a tensor product of a line bundle (namely $(\omega_{\sX_1/\sH_1}^{[k'm]}\otimes\sI_\sD^{[k\frac{m}{r}]})^{[1]}$) and a relatively $S_2$
sheaf which commutes with base change (namely $(\omega_{\sX_1/\sH_1}^{[b']}\otimes\sI_\sD^{[b]})^{[1]}$).

{\bf Step 8:} In this step, we shrink $\sH_2$ to an open subset parametrizing semi-log canonical pairs.

By the reductions in the previous steps, for every $k \geq 1$, the ideal sheaf $\sI_{\sD_2}^{[k]}$ determines a family of generically Cartier divisors, in the sense of \cite{kol_new}*{Definition 4.24}.
Indeed, this is obtained by dualizing the inclusion $\sI_{\sD_2}^{[k]} \rar \O \sX_2.$, as observed in \cite{kol_new}*{Definition 4.24}.
Then, as observed in \cite{kol_new}*{4.25}, a family of generically Cartier divisors induces a well-defined family of divisors.
Abusing notation, we denote by $\sD_2$ the family of divisors corresponding to $\sI_{\sD_2}$.

Now, recall that taking the reduced structure does not change the topology.
Thus, in the remainder of this step, we may assume that $\sH_2$ is reduced.
Then, $(\sX_2,\frac{1}{r}\sD_2) \rar \sH_2$ is a well defined family of pairs, and by \cite{kol_new}*{Corollary 4.45}, there is an open locus where the fibers of $(\sX_2,\frac{1}{r}\sD_2) \rar \sH_2$ are semi-log canonical.
Also, since we took an open subset of $\sH_2$, this choice is not affected by considering the reduced structure of $\sH_2$.

{\bf Step 9:} In this step, we use Koll\'ar's theorem \ref{teorema kollar constructible} to shrink $\sH_2$ to a locally closed and functorial subscheme parametrizing the $\bQ$-stable pairs in our moduli problem.

Using Theorem \ref{teorema kollar constructible}, up to replacing $\cH_2$ with such a (functorial) locally closed subscheme of it, we can assume that $\omega_{\sX_2/\sH_2}^{[m]}\otimes \sI_{\sD_2}^{[-\frac{m}{r}]}$ is relatively semi-ample.
Similarly, since being ample is an open condition, we can assume that $\omega_{\sX_2/\sH_2}^{[m]} \otimes \sI_{\sD_2}^{[-\frac{m(1-\epsilon)}{r}]}$ is relatively ample.
Now, recall that we have chosen $\epsilon = \frac{1}{k}$ for a chosen fixed $k \in \mathbb N$.
Notice that, the two former conditions guarantee that, for every $j \geq k$, we have that $\omega_{\sX_2/\sH_2}^{[rjm]} \otimes \sI_{\sD_2}^{[-m(j-1)]}$ is relatively ample, as we have $mj > m(j-1) \geq mj(1-\frac{1}{k})$.
Then, we may fix $n+1$ such values $j_1,\ldots,j_{n+1}$ and disregard all components of $\sH_2$ but the ones where, over $s \in \sH_2$, $(\omega_{\sX_2/\sH_2}^{[rj_im]} \otimes \sI_{\sD_2}^{[-m(j_i-1)]})^n$ has prescribed value.
By flatness, these self-intersections are locally constant, and thus this condition is open.
For each $i$, the self-intersection is prescribed by $p(1-\frac{1}{j_i})$, up to the rescaling factor given by $rj_im$.
Since we are prescribing $n+1$ values of a polynomial of degree $n$, this guarantees that all the fibers correspond to $\bQ$-stable pairs $(X;D)$ with $(K_X+tD)^n=p(t)$.

{\bf Step 10:} In this step, we cut $\sH_2$ to a closed subset $\sH_3$ to ensure that the natural polarization $\O \sH_3 \times \mathbb{P}^N.(1)$ coincides with $\omega_{\sX_3/\sH_3}^{[m]} \otimes \sI_{\sD_3}^{[-\frac{m(1-\epsilon)}{r}]}$.

By construction, for every $\bQ$-stable pair $(X;D)$ in our moduli problem, we have $\omega_{\sX_3/\sH_3}^{[m]} \otimes \sI_3^{-[\frac{m(1-\epsilon)}{r}]}|_X \sim \O X. (m(\K X. + (1-\epsilon)D)) \sim \O \mathbb P ^N .(1)|_X$.
Since $\omega_{\sX_3/\sH_3}^{[m]} \otimes \sI_3^{-[\frac{m(1-\epsilon)}{r}]}$ is a line bundle and the natural polarization of $\sH_2$ coming from the original choice of Hilbert scheme restricts to $\O \mathbb P ^ N. (1)$ fiberwise, by \cite{Vie95}*{Lemma 1.19} there is a locally closed subscheme $\sH_3$ where $\omega_{\sX_3/\sH_3}^{[m]} \otimes \sI_3^{-[\frac{m(1-\epsilon)}{r}]}$ is linearly equivalent to the natural polarization of $\sX_2 \rar \sH_2$ over $\sH_3$.

{\bf Step 11:} In this step, we show that there is an isomorphism $\sF_{n,p,I}\cong [\sH_3/\PGL_{N+1}] $.

Our argument follows closely \cite{Alp21}*{Theorem 2.1.11}.
First, we observe that all the cuts performed in the previous steps depend on properties that are invariant under the action of $\PGL_{N+1}$, thus the natural action of $\PGL_{N+1}$ descends onto $\sH_3$.
Then, observe that from its construction, over $\sH_3$ there is a stable family of $\bQ$-stable pairs.
This gives a morphism $\sH_3 \to \sF_{n,p,I}$, and if we forget the embedding into $\bP^N$, this descends to a morphism $\Phi^{pre} \colon [\sH_3/\PGL_{N+1}]^{pre} \to \sF_{n,p,I}$, where the superscript pre stands for prestack (see \cite{Alp21}*{Definition 1.3.12}).
This induces a map $\Phi \colon [\sH_3/\PGL_{N+1}] \to \sF_{n,p,I}$, which we now show is an isomorphism.

To show it is fully faithful, as in \cite{Alp21}, it suffices to check that $\Phi^{pre}$ is fully faithful.
But $\Phi^{pre}$ is fully faithful since any isomorphism between two families of $\bQ$-stable
pairs $\pi \colon (\sY;\sD)\to B$ and $\pi \colon (\sY',\sD')\to B$ over $B$ sends $\sL \coloneqq \omega_{\sY/B}^{[m]} \otimes \sI_{\sY}^{-[\frac{m(1-\epsilon)}{r}]}$ to $\sL' \coloneqq\omega_{\sY'/B}^{[m]} \otimes \sI_{\sY'}^{-[\frac{m(1-\epsilon)}{r}]}$, where we denoted by $ \sI_{\sY}$ (resp. $\sI_{\sY'}$) the ideal sheaves of $\sD$ (res. $\sD'$) in $\sY$ (resp. $\sY'$). This induces an unique isomorphism $\bP(\pi_*\sL )\cong \bP(\pi'_*\sL')$ which sends $\sY$ to $\sY'$.

Since $\Phi$ is a morphism of stacks, also essential surjectivity can be checked locally on $B$.
In particular, it suffices to check that if $\pi \colon (\sY;\sD)\to B$ is a family of $\bQ$-stable pairs such that $\pi_*\sL$ is free, then the morphism $B\to \sF_{n,p,I}$ lifts to a morphism $B\to \sH_3$.
This follows since if we pick an isomorphism $\bP (\pi_*\sL)\cong \bP^N \times B$ then
$\sY, \sD\subseteq \bP (\pi_*\sL) = \bP^N\times B$, and then from the functorial properties of $\sH_3$ it induces a morphism $B\to \sH_3$.
\end{proof}
\begin{Remark}
Observe that the stack $\sF_{n,p,I}$ is in fact Deligne--Mumford.
Indeed, since we are working over a field of characteristic 0, it suffices to show that the automorphisms of the objects over the points are finite. But this follows since an automorphism of a $\bQ$-pair $(X;D)$ induces an automorphism of the stable pair $(X,(1-\epsilon)D)$, and those are finite from \cite{KP17}*{Proposition 5.5}.
\end{Remark}

\section{Properness of $\sF_{n,p,I}$}\label{section_properness}
The goal of this section is to prove that $\sF_{n,p,I}$ is proper.
In particular, since in the definition of a $\bQ$-stable pair $(X;D)$ there are prescribed conditions on the scheme-theoretic structure of $D$.
Then, when proving that a moduli functor for $\bQ$-pairs satisfies the valuative criterion for properness, one needs to check that these scheme-theoretic properties are preserved.
It is convenient to check that the flat limit of $D^{sc}$ (recall that $D^{sc}$ was introduced in Notation \ref{notation_Dsc}) is $S_1$.
This will be the content of the next proposition.

\begin{Prop}\label{prop:flat:limit:is:S1}
Let $\spec(R)$ be the spectrum of a DVR with generic point $\eta$ and closed point $p$.
Consider a locally stable family $(X,D)\to \spec(R)$ such that $D$ is $\mathbb{Q}$-Cartier.
Then, for every $m, m'\in \mathbb{N}$, the sheaf $\omega_{X/B}^{[m']}(-mD)$ is $S_3$ on every point $x\in X_p$.
In particular:
\begin{enumerate}
    \item the restriction $\omega_{X/B}^{[m']}(-mD)|_{X_p}$ is $S_2$; and
    \item if we denote by $mD$ the closed subscheme of $X$ with ideal sheaf $\cO_X(-mD)$, then $(mD)|_{X_p}$ is $S_1$.
\end{enumerate}
\end{Prop}

\begin{proof}
First, observe that since $X\to \spec(R)$ is locally stable, $x \in X_p$ cannot be a log canonical center for $X$ (see \cite{kol_new}*{Proposition 2.15}).
The statement is local, so up to shrinking $X$, we can assume that $X$ is affine, and, since both $\omega_{X/B}$ and $D$ are $\bQ$-Cartier, $\omega_{X/B}^{[m']}(-mD) \cong \cO_X$ for a certain $m_0 \in \nn$.
Then, for every $m, m' \in \nn$, we have  $\omega_{X/B}^{[m_0m']}(-m_0mD)\sim_\bQ0$.
In particular, if we apply \cite{Kol13}*{Theorem 7.20} where, with the notations of \emph{loc. cit.}, we take $\Delta'=0$ and if we denote by $D_K$ the divisor $D$ in \emph{loc. cit.}, $D_K= - m'K_{X/B} + mD$, we conclude that $\omega_{X/B}^{[m']}(-mD)$ is $S_3$.

Now, we denote by $\pi$ the pull-back to $X$ of a uniformizer on $\spec(R)$.
Since $\omega_{X/B}^{[m']}(-mD)$ is $S_2$ and $\pi$ is not a zero divisor on $\cO_X$, it is not a zero divisor for $\omega_{X/B}^{[m']}(-mD)$.
Then, \begin{center}$(\omega_{X/B}^{[m']}(-mD))|_{X_p}$ is $S_2$ by \cite{KM98}*{Proposition 5.3}.\end{center}

Notice that $\cO_X$ is $S_3$ along the special fiber, since it is a flat and proper family of $S_2$ schemes over a smooth base (see \cite{KM98}*{Proposition 5.3}).
Then, by \cite{Kol13}*{Corollary 2.61}, $\cO_{mD}$ is $S_2$ along the special fiber.
In particular, it is $S_1$, so its generic points are the only associated points.
Hence, if we denote by $mD$ the closed subscheme of $X$ with ideal sheaf $\cO_X(-mD)$, then $mD\to \spec(R)$ is flat.
Then, if we pull back the exact sequence
\[
0\to \cO_X(-mD)\to \cO_X\to \cO_{mD}\to 0
\]
to $X_p$, the sequence remains exact and we get
\[
0\to (\cO_X(-mD))|_{X_p}\to (\cO_X)|_{X_p}\to (\cO_{mD})|_{X_p}\to 0.
\]
The desired result follows again by \cite{Kol13}*{Corollary 2.61}.
\end{proof}

\begin{Prop} \label{properness_lc_case}
Fix an integer $n\in \mathbb{N}$, a finite subset $I\subseteq (0,1] \cap \mathbb{Q}$, and a polynomial $p(t) \in \mathbb{Q}[t]$.
Assume that $I$ is closed under sum: that is, if $a,b \in I$ and $a+b \leq 1$, then $a+b \in I$.
Let $C$ be a smooth affine curve, let $0 \in C$ be a distinguished closed point, and let $U\coloneqq C \setminus \{ 0 \}$.
Let $(\cX_U, \Delta_U; \cD_U) \rightarrow U$ be a weak $\bQ$-stable morphism of relative dimension $n$, polynomial $p(t)$, coefficients in $I$, and constant part $\Delta_U$.
Further, assume that the geometric generic fiber is normal.
Then, $(\cX_U, \Delta_U; \cD_U) \rightarrow U$ is a $\bQ$-stable morphism, and, up to a finite base change $B \rar C$, the fibration can be filled with a unique $\bQ$-stable pair of dimension $n$, with polynomial $p(t)$ and coefficients in $I$.
\end{Prop}

\begin{proof}
We proceed in several steps.
In the following, $r$ will denote the index of $I$.

{\bf Step 1:} In this step, we show that the family of pairs $(\cX_U,\frac{1}{r}\cD_U + \Delta_U) \rar U$ admits a relative canonical model $(\cY_U, \frac{1}{r}\cD _{U,\cY} + \Delta_{U, \cY})$ over $U$.

By Proposition \ref{prop_stable_family_p-pairs}, the fact that $C$ is smooth and inversion of adjunction, it follows that $(\cX _U, \frac{1}{r}\cD_U + \Delta_U)$ is a log canonical pair.
By Lemma \ref{lemma minimal models} and the fact that, by definition, all the fibers admit a good minimal model, the pair $(\cX _U, \frac{1}{r}\cD_U + \Delta_U)$ admits a morphism to its relative canonical model $(\cY_U, \frac{1}{r}\cD _{U,\cY} + \Delta_{U, \cY})$ over $U$.

{\bf Step 2:} In this step, we show that, up to a base change, $(\cY_U, \frac{1}{r}\cD _{U,\cY} + \Delta_{U,\cY}) \rar U$ can be compactified to a family of stable pairs $(\cY,\frac{1}{r}\cD_\cY + \Delta_\cY)$.

This is accomplished in \cite{kol_new}*{Theorem 4.59}.

{\bf Step 3:} In this step, we choose a suitable $\qq$-factorial dlt modification $(\cX', \frac{1}{r}\cD' + \Delta'+\cX'_0) \rightarrow (\cY, \frac{1}{r}\cD _{\cY} + \Delta_{\cY}+\cY_0)$ so that $\cX'_U \drar \cX_U$ is a rational contraction over $U$.

Since $K_{\cX_U}+\frac{1-\epsilon}{r}\cD_U + \Delta_U$ is ample over $U$ for $0 < \epsilon \ll 1$, the divisors contracted by $\cX_U \rightarrow \cY_U$ are contained in the support of $\cD$.
Then, by \cite{Mor19}*{Theorem 1}, we can extract all of these divisors.
Finally, we take a $\qq$-factorial dlt modification of this model just constructed.
We denote this model by $(\cX', \frac{1}{r}\cD' + \Delta'+\cX'_0) \rightarrow (\cY, \frac{1}{r}\cD _{\cY} + \Delta_{\cY}+\cY_0)$.
In writing $(\cX', \frac{1}{r}\cD' + \Delta'+\cX'_0)$, we define $\cD'$ by first pulling back $\cD_U$ to a common resolution of $\cX'_U$ and $\cX_U$, then pushing it forward to $\cX'_U$, and finally taking its closure in $\cX'$.
In particular, $\cD'$ is horizontal over $C$.

{\bf Step 4:} In this step, we construct the compactification $(\cX,\Delta;\frac{1}{r}\cD)$.

We consider the pair $(\cX', \frac{1-\epsilon}{r}\cD' + (\Delta')^h)$ for $0<\epsilon \ll 1$, where the notation $(\Delta')^h$ stands for horizontal over $C$.
Then, $(\cX_U,\frac{1-\epsilon}{r}\cD_U+\Delta_U)$ is a relative good minimal model for $(\cX'_U, \frac{1-\epsilon}{r}\cD'_U + (\Delta'_U)^h)$ over $\cY_U$.
Then, by \cite{HX13}, $(\cX', \frac{1-\epsilon}{r}\cD' + (\Delta')^h)$ admits a relative good minimal model over $\cY$.
We denote its relative canonical model over $\cY$ by $(\cX,\frac{1-\epsilon}{r}\cD+\Delta^h)$.

First, we observe that $\Delta^h=\Delta$.
Notice that $\Delta_U$ is horizontal over $U$ by assumption, so $\Delta^v$, which denotes the vertical components of $\Delta$, has to be supported on $\cX_0$.
Then, the claim follows, since $(\cY,\frac{1}{r}\cD_\cY+\Delta_\cY + \cY_0)$ is log canonical.

Now, we observe that $\cD$ is $\qq$-Cartier.
Indeed, $\K \cX. + \frac{1}{r}\cD+\Delta$ is $\qq$-Cartier, as it is the pull-back of $\K \cY. + \frac{1}{r}\cD_\cY + \Delta_\cY$.
By construction, we have that $\K \cX. + \frac{1-\epsilon}{r}\cD + \Delta$ is ample over $\cY$, and in particular it is $\qq$-Cartier.
Then, $\cD$ is the difference of two $\qq$-Cartier divisors.

Since $\cD$ is $\qq$-Cartier, it follows that this canonical model is independent of $0 < \epsilon \ll 1$, as $\K \cX. + \frac{1-\epsilon}{r}\cD + \Delta \sim_{\qq,\cY} -\frac{\epsilon}{r}\cD$.
Furthermore, if $0 < \epsilon \ll 1$, we have that $\K \cX. + \frac{1-\epsilon}{r}\cD + \Delta$ is ample over $C$, as it is ample over $\cY$ and $\K \cY. + \frac{1}{r}\cD_\cY + \Delta_\cY$ is ample over $C$.

{\bf Step 5:} In this step, we show that $(\cX,\Delta;\frac{1}{r}\cD)$ is a $\bQ$-pair and that the central fiber of $(\cX,\Delta;\frac{1}{r}\cD)$ is a $\bQ$-stable pair with polynomial $p(t)$ and coefficients in $I$.

Recall that $\cX$ is independent of $\epsilon$, and that $\cD$ is $\qq$-Cartier.
As $(\cX,\frac{1}{r}\cD + \Delta)$ is log canonical by construction, it follows that $(\cX,\Delta;\frac{1}{r}\cD)$ is a $\bQ$-pair.

By construction and by adjunction, $(\cX_0,\frac{1}{r}\cD_0 + \Delta_0)$ is semi-log canonical.
Furthermore, $\K \cX_0. + \frac{1-\epsilon}{r}\cD_0 + \Delta_0$ is ample for $0 < \epsilon \ll 1$.
Since $\cD$ and $K_{\cX} + \frac{1}{r}\cD +\Delta$ are $\mathbb{Q}$-Cartier, the self-intersection $(K_{\cX_c} + \frac{1}{r}\cD_c +\Delta_c-\epsilon \cD_c)^{{\rm dim}(\cX_0)}$ is well-defined for every $\epsilon$ and $c\in C$, and does not depend on $c \in C$.
Since the general fiber is $\bQ$-stable with polynomial $p(t)$, we have $(K_{\cX_0} + \frac{1}{r}\cD_0 +\Delta_0-\epsilon \cD_0)^{{\rm dim}(\cX_0)} = p(1-\epsilon)$.
The coefficients of $\frac{1}{r}\cD_0$ are still in $I$,  since $I$ is closed under addition.

{\bf Step 6:} In this step, we show that $(\cX, \Delta;\cD) \rar C$ is a $\bQ$-stable morphism with constant part $\Delta$.

By construction, the fibers are proper.
Since the base is a curve and every divisor is horizontal, all the morphisms are flat of the appropriate relative dimension.
By Step 7, every fiber is a $\bQ$-stable pair.
Then, as $\cD$ is $\qq$-Cartier, every fiber of $(\cX,\Delta) \rar C$ is semi-log canonical.
Thus, by \cite{kol_new}*{Definition-Theorem 4.7}, the morphism $(\cX,\Delta) \rar C$ is locally stable.

{\bf Step 7:} In this step, we show that the limit is unique.

From Theorem \ref{thm_boundedness}, there exists $\epsilon_0 > 0$ such that, for every $\bQ$-stable pair $(Z,B,\Gamma)$ with coefficients in $I$ and polynomial $p(t)$, $(Z,(1-\epsilon_0)B + \Gamma)$ is a stable pair.
Then, the claim follows from the separatedness of stable morphisms \cite{kol_new}*{2.49}.
\end{proof}

\begin{theorem}\label{Thm_properness}
Fix an integer $n\in \mathbb{N}$, a finite subset $I\subseteq (0,1] \cap \mathbb{Q}$, and a polynomial $p(t) \in \mathbb{Q}[t]$.
Assume that $I$ is closed under sum: that is, if $a,b \in I$ and $a+b \leq 1$, then $a+b \in I$.
Let $C$ be a smooth affine curve, let $0 \in C$ be a distinguished closed point, and let $U\coloneqq C \setminus \{ 0 \}$.
Let $(\cX_U, \Delta_U; \cD_U) \rightarrow U$ be a weak $\bQ$-stable morphism of dimension $n$, polynomial $p(t)$, coefficients in $I$, and constant part $\Delta_U$.
Then, $(\cX_U, \Delta_U; \cD_U) \rightarrow U$ is a $\bQ$-stable morphism, and, up to a finite base change $B \rar C$, the fibration can be filled with a unique $\bQ$-stable pair of dimension $n$, with polynomial $p(t)$ and coefficients in $I$.
\end{theorem}

\begin{proof}
By Proposition \ref{properness_lc_case}, we may assume that the geometric generic fiber is not normal.
In the following, $r$ will denote the index of $I$.

Let $(\overline \cX_U, \frac{1}{r}\overline \cD_U + \overline \Delta_U)$ denote the normalization of $(\cX_U, \frac{1}{r}\cD_U + \Delta_U)$, where $\overline{\Delta}_U$ also includes the conductor with coefficient 1.
By assumption, $\cD_U$ is $\bQ$-Cartier.
Then, by \cite{Kol13}*{Corollary 5.39}, $\overline{\cD}_U$ is $\bQ$-Cartier.
Then, since $\K \cX_U. + \frac{1-\epsilon}{r}\cD_U + \Delta_U$ is ample over $U$, it follows that $(\overline \cX _U, \overline \Delta_U; \overline \cD _U)$ is $\bQ$-stable morphism with constant part $\overline \Delta_U$, where the polynomial on each connected component depends on the original choice of $p$.

Then, as $\overline{\cX}_U$ has finitely many connected components, by Proposition \ref{properness_lc_case}, there is a finite base change $B \rar C$ such that the family can be filled with a unique $\bQ$-stable pair.
To simplify the notation, we omit the base change $B \rar C$, and we assume that the filling is realized over $C$ itself.
We denote this family by $(\overline \cX,  \overline \Delta; \overline \cD) \rar C$.
Then, we may find $0 < \epsilon \ll 1$ such that $(\overline \cX, \frac{1-\epsilon}{r} \overline \cD  + \overline \Delta) \rar C$ is a stable morphism.
By \cite{kol_new}*{Lemma 2.54}, also $(\cX _U, \frac{1-\epsilon}{r} \cD _U + \Delta_U)$ admits a compactification $(\cX , \frac{1-\epsilon}{r} \cD + \Delta)$ over $C$ obtained by gluing $\overline{\cX}$ along some components of $\overline{\Delta}$.
By \cite{Kol13}*{Corollary 5.39}, the divisor $\cD$ is $\bQ$-Cartier.
Thus, we have that $\cD_0$ is $\bQ$-Cartier, as needed.
Similarly, the coefficients of $\frac{1}{r}\cD_0+\Delta_0$ are in $I$, by construction.

Now, we prove that the special fiber $(\mathcal{X}_0,\Delta_0;\frac 1 r \mathcal{D}_0)$ is a $\mathbb Q$-stable pair.
We already verified that $\mathcal{D}_0$ is $\mathbb Q$-Cartier.
Now, we verify that $(\mathcal{X}_0,\Delta_0;\frac 1 r \mathcal{D}_0)$ is semi-log canonical.
By construction, $(\mathcal{X}_0,\frac{1-\epsilon}{r} \mathcal{D}_0+\Delta_0)$ is semi-log canonical.
Thus, by \cite{Kol13}*{Definition-Lemma 5.10}, it suffices to show that the normalization of $(\mathcal{X}_0,\frac{1}{r} \mathcal{D}_0+\Delta_0)$ is log canonical.
But then, this holds by construction, as its normalization coincides with the normalization of the (possibly disconnected) semi-log canonical pair $(\overline{\mathcal{X}}_0,\frac{1}{r}\overline{\mathcal{D}}_0+\overline{\Delta}_0)$.
The same argument applied to the total space shows that $(\overline{X},\frac{1}{r}\mathcal{D}+\Delta)$ is semi-log canonical.

Lastly, we need to check that $K_{\mathcal{X}_0}+\frac{1}{r} \mathcal{D}_0+\Delta_0$ is semi-ample.
To this end, by adjunction, it suffices to show the stronger statement that $K_{\mathcal{X}}+\frac{1}{r} \mathcal{D}+\Delta$ is semi-ample over $C$.
By Proposition~\ref{properness_lc_case}, this is true for $K_{\overline{\mathcal{X}}}+\frac{1}{r} \overline{\mathcal{D}}+\overline{\Delta}$.
Then, the claim follows by \cite{HX16}*{Theorem 2}.
Then, since $\mathcal{D}$ does not contain any irreducible component of the double locus of $\mathcal{X}$, it is immediate that the relative canonical model of $(\mathcal{X},\frac{1}{r}\mathcal{D}+\Delta)$ coincides with the gluing of the relative canonical models of the irreducible components of $(\overline{\mathcal{X}},\frac{1}{r}\overline{\mathcal{D}}+\overline{\Delta})$, which in turn provides a stable family.
We observe that the gluing of said ample models of is possible by \cite{HX13}*{\S~7} and the following fact: as the exceptional locus of the morphism from $\overline{\mathcal{X}}$ to the canonical model is contained in $\overline{\mathcal{D}}$ and $\overline{\mathcal{D}}$ does not contain any irreducible component of the conductor, the involution defined on the conductor via the normalization of $(\mathcal{X},\frac{1}{r}\mathcal{D}+\Delta)$ naturally descends to the canonical model of $(\overline{\mathcal{X}},\frac{1}{r}\overline{\mathcal{D}}+\overline{\Delta})$.
We observe that the involution defined on the conductor via the normalization of $(\mathcal{X},\frac{1}{r}\mathcal{D}+\Delta)$ preserves the different, which is a necessary condition for the gluing of the ample models, by \cite{Kol13}*{Proposition 5.12}.

To conclude, we need to show that $(\K \cX_0. + \frac{t}{r}\cD_0 + \Delta_0)^n=p(t)$.
This follows from flatness over the base $C$, as we have
\[
(\K \cX_0. + \frac{t}{r}\cD_0 + \Delta_0)^n = (\K \cX_c. + \frac{t}{r}\cD_c + \Delta_c)^n = p(t),
\]
where $c \in C \setminus \{ 0 \}$.
This concludes the proof.
\end{proof}

\begin{Cor}Fix an integer $n\in \mathbb{N}$, a finite subset $I\subseteq (0,1] \cap \mathbb{Q}$, and a polynomial $p(t) \in \mathbb{Q}[t]$.
Assume that $I$ is closed under sum: that is, if $a,b \in I$ and $a+b \leq 1$, then $a+b \in I$. Then the algebraic stack $\sF_{n,p,I}$ is proper.
\end{Cor}
\begin{proof}
It suffices to check that it satisfies the valuative criterion for properness. So assume that we have a family of $\bQ$-stable pairs $f_\eta \colon (X;D)\to \eta$ over the generic point $\eta$ of the spectrum of a DVR $R$, and we need to fill in the central fiber up to replacing $\spec(R)$ with a ramified cover of it. Theorem \ref{Thm_properness} guarantees the existence and uniqueness of a $\bQ$-stable morphism $f \colon (\cX;\cD)\to \spec(R)$ extending $f_\eta$, up to a ramified cover of $\spec(R)$. To check that $f$ satisfies condition (K) of Definition \ref{Def:functor} we use Proposition \ref{prop:flat:limit:is:S1} point (1).
\end{proof}

\section{Relative canonical models over reduced base}
\label{section morphism}
Given a pair $(X,D)$ with $K_X+D$ semi-ample and big, one can take its canonical model.
Similarly, from Definition \ref{def families A+}, if one starts with a family of $\bQ$-stable pairs, one can take this canonical model in families.
The goal of this section is to show that, \textit{over a reduced base}, the condition in Definition \ref{def families A+} is a fiberwise condition:
weak families of $\bQ$-stable pairs are actually families of $\bQ$-stable pairs.

\begin{Lemma}\label{Lemma:slc:canonical:model:for:p:pairs}
Fix an integer $n\in \mathbb{N}$, a finite subset $I\subseteq (0,1] \cap \mathbb{Q}$, and a polynomial $p(t) \in \mathbb{Q}[t]$.
Let $r$ denote the index of $I$.
Consider a weak $\bQ$-stable morphism with coefficients in $I$ and polynomial $p(t)$ over a smooth scheme $U$:
$p \colon (\cX;\cD)\to U$.
Then, there is a stable family of pairs $(\cY,\frac{1}{r}\cD_\cY)\to U$ such that:
\begin{enumerate}
    \item there is a contraction $\pi \colon \cX \to \cY$; and
    \item we have $\pi^* (\K \cY/U. + \frac{1}{r}\cD_\cY) = \K \cX /U. + \frac{1}{r}\cD$.
\end{enumerate}
In particular, $p \colon (\cX;\cD)\to U$ is a $\bQ$-stable morphism.
\end{Lemma}
\begin{proof}
By Proposition \ref{prop_stable_family_p-pairs} and the fact that $U$ is smooth, it follows that $(\cX, \frac{1}{r}\cD)$ is a pair.
Furthermore, by inversion of adjunction, this pair is semi-log canonical.

First, we assume that $\cX$ is normal.
Then, by Lemma \ref{lemma minimal models} and our assumptions, the canonical model of $(\cX,\frac{1}{r}\cD)$ over $U$ exists.
Since $U$ is smooth, this canonical model gives a family of stable pairs by \cite{kol_new}*{Corollary 4.57}, and the result follows from how canonical models are constructed.

We now treat the case where $\cX$ is not normal.
First, we normalize $n_\cX \colon \cX^n\to \cX$ to get $(\cX^n,\Delta;\cD') $.
As argued in the previous case, we can construct the canonical model of $(\cX^n,\Delta;\cD') $ over $U$.
As above, this gives a family of stable pairs $(\cY',\frac{1}{r}\cD_\cY'+\Delta_\cY)\to U$.
From Koll\'ar's gluing theory (see \cite{Kol13}*{Ch. 5}), there is an involution $\tau \colon \Delta^n \to \Delta^n$ that fixes the different.
We first show that this involution descends onto $\Delta_\cY^n$.

Recall that, by Lemma \ref{ample model p-pair}, the map $\cX^n\to \cY'$ does not contract any component of $\Delta$.
In particular, for every irreducible component $F\subseteq \Delta^n$, there is an irreducible component $F_Y\subseteq \Delta_\cY^n$ birational to it, and we have the following diagram:

$$\xymatrix{\Delta^n\ar[r]^f \ar[d] & \Delta^n_\cY \ar[d] \\ \cX^n \ar[r]^{\pi_n} & \cY'.}$$
Observe now that
\begin{center}
    {\bf ($\ast$)} {\it a curve $C \subset \Delta^n$ gets contracted by $f$ if and only if $(K_{\Delta^n} + \operatorname{Diff}_{\Delta^n}(\frac{1}{r}\cD ' + \Delta)).C = 0$.}
\end{center}
In particular, since the involution $\tau$ preserves the different, it preserves all the curves that are contracted by $f$.
Hence, the involution descends to an involution $\tau_\cY$ on $\Delta_\cY^n$.

We prove that $\tau_\cY$ preserves the different.
Indeed, by Lemma \ref{ample model p-pair}, the only divisors contracted by $f$ are contained in $\supp(\cD')$, so the morphism $f$ is an isomorphism generically around each divisor not contained in $\supp(\cD)$.
In particular, since the computation of the different is local,
\[
f_*(\operatorname{Diff}_{\Delta^n}(\frac{1}{r}\cD' + \Delta)) = \operatorname{Diff}_{\Delta^n_\cY}(\frac{1}{r}\cD'_\cY + \Delta_\cY).
\]
Since $\tau$ preserves the different on $\Delta^n$, $\tau_\cY$ preserves the different on $\Delta_\cY^n$.

Then from \cite{Kol13}*{Ch. 5}, we can glue $\cY'$ to get an semi-log canonical pair $(\cY,\frac{1}{r}\cD_\cY)$, and let $n_\cY \colon \cY'\to \cY$ be the normalization.
Now, recall that $\cY$ is a geometric quotient (see \cite{Kol13}*{Theorem 5.32} and the proof of \cite{Kol13}*{Corollary 5.33}), so for any morphism $\phi \colon \cY^n\to Z$ such that $\phi|_{\Delta_\cY^n} = \tau \circ \phi|_{\Delta_\cY^n} $, there is a unique morphism $\cY\to Z$ which fits in the obvious commutative diagram.
Therefore, we have a morphism $\cY\to U$, and applying this result to $\cX^n$ and $\cX$, we obtain a morphism $\pi \colon \cX\to \cY$.

To show (1) it suffices to observe the following two exact sequence: 
\begin{equation} \label{eq_ses_normalization}
0\to \cO_\cX \to (n_{\cX})_*\cO_{\cX^n} \xrightarrow{g_\cX} (n_{\cX})_*\cO_{\Delta^n}.
\end{equation}
Observe that $\pi\circ n_\cX = n_\cY\circ\pi_n $ and both $\pi_n$ and $\Delta^n\to \Delta^n_\cY$ is birational. So $(\pi_n)_*\cO_{\cX^n} =\cO_{\cY'} $ and $(\pi\circ n_\cX)_* \cO_{\Delta^n} = \cO_{\Delta^n_\cY}.$ Therefore pushing forward the sequence \eqref{eq_ses_normalization} via $\pi$ we have
$$
0\to \pi_*\cO_\cX \to (n_{\cY})_*\cO_{\cY'} \xrightarrow{g_\cX} (n_{\cX})_*\cO_{\Delta^n_\cY}.
$$
In particular, $\pi_*\cO_{\cX}=\cO_{\cY}$. Similarly we can tensor the sequence above by $\pi^*(K_{\cY/U} + \frac{1}{r}\cD)$ to deduce (2).
\end{proof}

\begin{theorem}\label{thm_morphism_our_stack_to_kollars_one_for_reduced_bases}
Fix an integer $n\in \mathbb{N}$ and a finite subset $I\subseteq (0,1] \cap \mathbb{Q}$.
Let $r$ denote the index of $I$.
Assume that the fibers of $p$ have an canonical model.
Let $(\cX;\cD)\to B$ be a weak $\bQ$-stable morphism of dimension $n$, with coefficients in $I$, and polynomial $p(t)$, over a reduced connected base $B$, and assume that there is an open dense subscheme $U\subseteq B$ with:
\begin{enumerate}
    \item a stable family of pairs $(\cY_U,\frac{1}{r}\cD_{\cY,U})\to U$ of relative dimension $n$; and
    \item a contraction $\pi_U \colon \cX_U\to \cY_U $ such that $\pi^*(K_{\cY_U/U} + \frac{1}{r}\cD_{\cY,U}) = K_{\cX_U/U} + \frac{1}{r}\cD_{U}$.
\end{enumerate}
Then, there there is an $m>0$ such that for every $d$ and every $b\in B$ we have $$p_*(\cO_{\cX}(md(K_{\cX/B}+\frac{1}{r}\cD)))\otimes k(b) = H^0(\cX_b, \cO_{\cX_b}(md(K_{\cX_b}+\frac{1}{r}\cD_b))) .$$

Moreover, if we define $\cY \coloneqq \operatorname{Proj}(\bigoplus_d p_*\cO_{\cX}(md(K_{\cX/B}+\frac{1}{r}\cD)))$, then:
\begin{itemize}
    \item there is a family of divisors $\cD_\cY$ such that the pair $q \colon  (\cY,\frac{1}{r}\cD_\cY) \to B $ is a stable family extending $q_U$;
    \item there is a contraction $\pi \colon \cX\to \cY$ over $B$ that extends $\pi_U$; and
    \item $\pi^*(K_{\cY/B} + \frac{1}{r}\cD_{\cY}) = K_{\cX/B} + \frac{1}{r}\cD$.
\end{itemize}
In particular, $(\mathcal X, \mathcal{D}) \rar B$ is a $\bQ$-stable morphism.
\end{theorem}
\begin{Remark}
Observe that since $K_{\cY_U/U} + \frac{1}{r}\cD_{\cY,U}$ is $\bQ$-Cartier, we can define its pull-back as a $\bQ$-Cartier divisor. 
\end{Remark}

\begin{proof}
We proceed in several steps.

{\bf Step 1:} In this step, we make some preliminary considerations and set some notation.

By Proposition \ref{prop_stable_family_p-pairs}, we have that $\K \cX/B. + \frac{1}{r}\cD$ and $\cD$ are $\qq$-Cartier.
Then, observe that for every $s,t\in U$, the volumes of the pairs $(\cY_s,\frac{1}{r}\cD_{\cY,s})$ and $(\cY_t,\frac{1}{r}\cD_{\cY,t})$ agree.
Indeed, since $(\cX;\cD)\to B$ is a weak $\bQ$-stable morphism, $\K \cX_s.+ \frac{1}{r}\cD_s$ and $\K \cX_t.+ \frac{1}{r}\cD_t$ are nef.
Thus, their volumes are computed by the $n$-fold self-intersection, which is independent of $s,t \in S$.
But from condition (2) the morphisms $\pi_s \colon \cX_s\to \cY_s$ and $\pi_t \colon \cX_t\to \cY_t$ have connected
fibers, and we have $\pi_s^*(\K \cY_s. + \frac{1}{r}\cD \subs \cY,s.)=\K \cX_s. +\frac{1}{r} \cD_s$ and $\pi_t^*(\K \cY_t. +\frac{1}{r} \cD \subs \cY,t.)=\K \cX_t. +\frac{1}{r} \cD_t$.
Then, by Remark \ref{remark:cohom:connected:implies:one:can:take:global:sections:after:pull:back}, the volumes of $\K \cY_s. + \frac{1}{r}\cD \subs \cY,s.$ (resp. $\K \cY_t. + \frac{1}{r}\cD \subs \cY,t.$) and $\K \cX_s. +\frac{1}{r} \cD_s$ (resp. $\K \cX_t. + \frac{1}{r}\cD_t$) agree.
Let then $v$ be the volume of any fiber of $q_U$ and let $k$ be a natural number such that $r$ divides $k$ and, for every stable pair $(Y,D)$ of dimension $n$, volume $v$ and coefficients in $I$, the line bundle $\cO_Y(k(K_Y + D))$ is very ample and the higher cohomologies of all of its natural multiples vanish.
Notice that $k$ exists by \cite{HMX18}*{Theorem 1.2.2}.
Then, we set $\cL \coloneqq  \O \cX. (k(K_{\cX/B} + \frac{1}{r}\cD))$.
Up to replacing $k$ with a multiple, we may further assume that $\cL$ is Cartier.

{\bf Step 2:} In this step, we show that the theorem holds if $B$ is a smooth curve.

From Lemma \ref{Lemma:slc:canonical:model:for:p:pairs}, we can construct a family of stable pairs $(\cZ,\frac{1}{r}\cD_\cZ)\to B$ with a contraction $\phi \colon \cX\to \cZ$.
First, observe that $(\cZ_U,\frac{1}{r}\cD_{\cZ,U})\cong (\cY_U,\frac{1}{r}\cD_{\cY,U})$ over $U$.
Indeed, consider the reflexive sheaves $\cL _\cZ \coloneqq \O \cZ. (k(\K \cZ. + \frac{1}{r}\cD _\cZ))$ and $\cL _{\cY,U} \coloneqq \O \cY_U. (k(\K \cY_U. + \frac{1}{r}\cD _{\cY,U}))$.
By construction, the fibers of $(\cZ,\frac{1}{r}\cD_\cZ) \rar B$ and $(\cY_U,\frac{1}{r}\cD \subs \cY,U.) \rar U$ belong to the moduli problem of stable pairs with volume $v$ and coefficients in the finite set $I$.
Thus, by \cite{kol_new}*{Theorem 5.8.(4)} and the choice of $k$, $\cL _\cZ$ and $\cL _{\cY,U}$ are line bundles.
In order to apply \cite{kol_new}*{Theorem 5.8.(4)}, notice that we know that $\cL _\cZ$ and $\cL _{\cY,U}$ are line bundles away from the exceptional locus of $\phi$ and $\pi_U$, which are big open subsets restricting to big open subsets fiberwise.

Since $\pi_U$ and $\phi_U$ have connected fibers, we have
\[
\pi_U^*\cL_{\cY,U} \cong \cL_U \cong \phi^*_U\cL_{\cZ,U}.
\]
But both $\pi_U$ and $\phi_U$ have connected fibers so, by the projection formula, for every $m \geq 1$, we have
\[
H^0(\cY_U,\cL_{\cY,U}^{\otimes m}) = H^0(\cX_U,\pi^*_U\cL_{\cY,U}^{\otimes m}) =H^0(\cX_U,\cL_{U}^{\otimes m}) = H^0(\cX_U,\phi^*_U\cL_{\cZ,U}^{\otimes m})  =  H^0(\cZ_U,\cL_{\cZ,U}^{\otimes m}).
\]
But $\cL_{\cZ,U}$ and $\cL_{\cY,U}$ are ample over $U$, so $\cZ_U$ and $\cY_U$ are isomorphic as $U$-schemes, as they are the relative Proj of the same sheaf of graded algebras.
In particular, the three final claims of the theorem hold if we consider the family $(\cZ,\frac{1}{r}\cD_\cZ)$.

Thus, we are left with proving that for every $d$ and every $b\in B$, we have $p_*(\cL^{\otimes d})\otimes k(b) = H^0(\cX_b, \cL^{\otimes d}_b) $.
But since $\phi^*\cL_\cZ = \cL$, we have $\phi_b^*(\cL_{\cZ,b}) = \cL_{b}$.
Thus, by Remark \ref{remark:cohom:connected:implies:one:can:take:global:sections:after:pull:back}, we have $h^0(\cX_b,\cL _b^{\otimes m}) = h^0(\cZ_b,\cL \subs \cZ,b.^{\otimes m})$ for all $m \geq 1$.
Then, the latter is locally constant from the assumptions on $k$, as by the vanishing of the higher cohomologies we have $h^0(\cZ_b,\cL \subs \cZ,b.^{\otimes m})= \chi(\cZ_b,\cL \subs \cZ,b.^{\otimes m})$, and the Euler characteristic of $\cL \subs \cZ,b.^{\otimes m}$ is independent of $b \in B$. Now the desired statement follows from \cite{Mum74}*{Corollary 2, page 50}.

{\bf Step 3:} In this step, we return to the general case and we show that for every $m$, the morphism $B\ni b\mapsto h^0(\cX_b,\cL|_{\cX_b}^{\otimes m})$ is constant and the algebra $\bigoplus_m H^0(\cX_b,\cL|_{\cX_b}^{\otimes m})$ is finitely generated.

Observe that the claim holds for every $p\in U$. Indeed, the following diagram commutes:
$$\xymatrix{\cX_b \ar[r] \ar[d]_{\pi_b} & \cX_U \ar[d] \\ \cY_b \ar[r] & \cY_U}$$
which from point (2) guarantees $\pi_b^*(\cL_{\cY,b}) = \cL_{b}$.
Then, in this case, we can conclude as at the end of Step 2.
Furthermore, the finite generation follows from the fact that $\cL_{b}$ is very ample and $\cY_b$ is the Proj of its associated graded ring.

Now, we treat the case when $b \not \in U$.
Consider a smooth curve $C$ with a map $C\to B$.
Assume that the generic point of $C$ maps into $U$.
Notice that any point $b \in B$ is contained in the image of such a curve.
Then, for any point $s\in C$, we have the following diagram:
$$\xymatrix{\cX_s\ar[r] \ar[d] & \cX_C \ar[r] \ar[d] & \cX \ar[d] \\ \{s\}\ar[r] & C \ar[r] & B.}$$
Since both squares are fibered squares, the big rectangle is a fibered square.
In particular, since we want to prove that $B\ni b\mapsto h^0(\cX_b,\cL|_{\cX_b}^{\otimes m})$
is constant, and since we know it is constant as long as $b\in U$, it suffices to check that for every such $C$ the functions $C\ni s\mapsto h^0(\cX_s,\cL|_{\cX_s}^{\otimes m})$ are constant.
Now, this follows by Step 2.
Similarly, the finite generation of $\bigoplus_m H^0(\cX_s,\cL|_{\cX_s}^{\otimes m})$ follows from Step 2.

{\bf Step 4:} In this step, we construct the model $\cY$ and the morphism $\pi$.

By Step 3 and cohomology and base change (see \cite{Mum74}), the sheaves $p_*(\cL^{\otimes m})$ commute with the restriction to points. In particular, the algebra $\mathcal{A}  \coloneqq  \bigoplus_{m} p_*(\cL^{\otimes m})$ is finitely generated since it is finitely generated when restricted to every point $b\in B$.
So we can consider $\cY \coloneqq \Proj_B(\cA)$.

As observed in Step 3, the pluri-sections of $\cL|_{\cX b}$ are deformation invariant.
As $\cL|_{\cX b}$ is semi-ample for every $b \in B$ by our assumptions, it follows that $\cL$ is relatively semi-ample.
In particular, we have a morphism $\cX \rar \cY$.
Furthermore, this implies the equality $\cY_b=\mathrm{Proj}(\bigoplus_m H^0(\cX_b,\cL_b^{\otimes m}))$ for every $b \in B$.

As already discussed, this construction commutes with base change.
In particular, for every $b\in B$, for checking properties of $\cY_b$ we can consider a smooth curve $C\to B$ which sends the generic point to $U$ and the special one to $b$, and first pull back $\cY$ to $C$ and then restrict it to $b$:
$$\xymatrix{\cY_b\ar[r] \ar[d] & \cY_C \ar[r] \ar[d] & \cY \ar[d] \\ \{b\}\ar[r] & C \ar[r] & B.}$$
The advantage is that now we can apply the results of Step 2 to $\cY_C$.

Since we have $\cL_U \cong \pi_U^* \cL_{\cY,U}$ and $\cL_{\cY,U}$ is relatively ample over $U$, it follows that $\cY \times_B U = \cY_U$ and that $\pi$ extends $\pi_U$.
Furthermore, since $B$ is reduced, the construction commutes with base change, and each fiber is reduced, it follows that $\cY$ is reduced.

{\bf Step 5:} In this step, we show that $\pi$ is a contraction and we construct the divisor $\cD_\cY$.

We first prove that $\pi$ is a contraction. We denote by $V\subseteq \cY$ the locus where $\pi^{-1}(V)\to V$ is an isomorphism. Then
it follows from Lemma \ref{ample model p-pair}:
\begin{center}
    {\bf ($\ast \ast$)} {\it for every fiber $\cY_b$, the complement of $V_b \coloneqq V \cap \cY_b$ has codimension at least 2 in $\cY_b$ and it does not contain any irreducible component of the conductor of $\cY_b$.}
\end{center}
Consider now the inclusion $i \colon \pi^{-1}(V)\to \cX$, which induces the injective map $0\to \cO_\cX \to i_*\cO_{\pi^{-1}(V)}$. We can push this sequence forward via $\pi$ and we obtain $0\to \pi_*\cO_\cX \to \pi_*i_*\cO_{\pi^{-1}(V)}$. But $\pi\circ i  \colon \pi^{-1}(V)\to \cY$ is the inclusion $j \colon  V\hookrightarrow Y$, so $\pi_*i_*\cO_{\pi^{-1}(V)} = j_*j^*\cO_\cY$ is reflexive from \cite{HK04}*{Corollary 3.7}, and it is isomorphic to $\cO_{\cY}$ from \cite{HK04}*{Proposition 3.6.2}. In particular, this gives an injective map $\pi_*\cO_\cX\to \cO_\cY$. One can check that this is the inverse of the canonical morphism $\cO_\cY\to \pi_*\cO_\cX$, so in particular the latter is an isomorphism.

Consider the ideal sheaf $\cI$ of $\cD$, and consider the inclusion
\[
0\to \cI \to \cO_\cX.
\]
Then, as $\pi$ is a contraction, if we push it forward via $\pi$, we get
\[0\to \pi_*\cI \to \cO_\cY.
\]
In particular, $\pi_* \cI$ is an ideal sheaf on $\cY$.
We denote by $\cS\subseteq \cY$ the closed subscheme with ideal sheaf $\pi_*\cI$.

A priori, $\cS$ may not be pure dimensional, however, consider the intersection between $\pi^{-1}(V)$ and the locus in $\cX$ where $\cD$ is $\bQ$-Cartier: we denote this locus with $W$.
Observe that, since $\pi^{-1}(V)\to V$ is an isomorphism, we can identify $W$ with a subset of $\cY$.
Moreover, since $\cD$ is Cartier on codimension one point of $\cX$ and $V$ is a big open subset, the locus $W$ contains all the codimension one points of $\cY_b$ for every $b$.
Then, we consider $\cS \cap W$, and we define $\cD_\cY$ to be the closure of $\cS \cap W$ in $\cY$.

{\bf Step 6:} In this step, we show that $q \colon (\cY,\frac{1}{r}\cD_\cY) \rar B$ is a well defined family of pairs.

By construction, it is immediate that $\cD_\cY$ is a relative Mumford divisor in the sense of \cite{kol19s}*{Definition 1}. Indeed, the three conditions of \cite{kol19s}*{Definition 1} are now clear, since they hold on $\cX$.

Thus, we just need to check that $q \colon \cY\to B$ is flat.
By pulling back $\cY\to B$ along a smooth curve through $U$, it follows from Step 1 that all the fibers of $q$ are reduced and equidimensional.
Thus, $q$ is an equidimensional morphism with reduced fibers over a reduced base, so \cite{kol_new}*{Lemma 10.58} applies.

{\bf Step 7:} In this step, we show that $q \colon (\cY,\frac{1}{r}\cD_\cY) \rar B$ is a stable family of pairs and $\pi^*(K_{\cY/B} + \frac{1}{r}\cD_{\cY}) = K_{\cX/B} + \frac{1}{r}\cD$.

Since $q \colon (\cY,\frac{1}{r}\cD_\cY) \rar B$ is a well defined family of pairs and its fibers belong to a prescribed moduli problem for stable pairs, we can argue as in the proof of Proposition \ref{prop_stable_family_p-pairs} to conclude that $K_{\cY/B} + \frac{1}{r}\cD_{\cY}$ is $\mathbb Q$-Cartier.
In particular, $q \colon (\cY,\frac{1}{r}\cD_\cY) \rar B$ is a stable family of pairs.

By construction, we have that $K_{\cY/B} + \frac{1}{r}\cD_{\cY} =\pi_*( K_{\cX/B} + \frac{1}{r}\cD)$.
Furthermore, we have that the equality $\pi^*(K_{\cY/B} + \frac{1}{r}\cD_{\cY}) = K_{\cX/B} + \frac{1}{r}\cD$ holds over $U$.
Then, by construction, all the exceptional divisors of $\pi$ dominate $B$, as they are contained in the support of $\cD$.
Thus, as $K_{\cY/B} + \frac{1}{r}\cD_{\cY}$ is $\mathbb Q$-Cartier and so $\pi^*(K_{\cY/B} + \frac{1}{r}\cD_{\cY})$ is well defined, it follows that $\pi^*(K_{\cY/B} + \frac{1}{r}\cD_{\cY}) = K_{\cX/B} + \frac{1}{r}\cD$.
\end{proof}

\begin{Lemma}\label{lemma:finitedness:paper:zsolt}
With the notation and assumptions of Theorem \ref{thm_morphism_our_stack_to_kollars_one_for_reduced_bases}, assume that $B$ is an affine curve and that there is a stable pair $(Y,D_Y)$ such that $(\cY,\frac{1}{r}\cD_\cY)\cong (Y\times B,D_Y\times B)$.
Then there are finitely many isomorphism classes of $\bQ$-stable pairs in the fibers of $(\cX;\cD)\to B$.
\end{Lemma}
\begin{proof}
The proof is analogous to the proof of \cite{ABIP}*{Claim 6.2}, we summarize here the most salient steps of the argument.

\begin{bf}Step 1:\end{bf} Using Koll\'ar's gluing theory and the fact that stable pairs have finitely many isomorphisms, up to normalizing and disregarding finitely many points on $B$, we can assume that $\cX$ (and therefore also $\cY$) is normal.
This is achieved in \cite{ABIP}*{Lemma 6.5}.
In particular, $(Y\times B,D_Y\times B)$ is the canonical model of $(\cX,\frac{1}{r}\cD)$.

\begin{bf}Step 2:\end{bf} We observe that the divisors contracted by $\cX\to Y\times B$ have negative discrepancies and can be extracted by a log resolution of the form $Z\times B\to Y\times B$, where $Z\to Y$ is a log resolution.
This is achieved in \cite{ABIP}*{Proposition 6.13}.

\begin{bf}Step 3:\end{bf} To conclude, we observe that all the fibers of $(\cX;\cD)\to B$ are isomorphic in codimension 2.

But two \emph{stable} pairs $(X_1, D_1)$ and $(X_2, D_2)$ which are isomorphic in codimension 2 must be isomorphic. Indeed, if $U$ is the open subset where they agree, and $L_1$ (resp. $L_2$) is the log-canonical divisor $K_{X_1} + D_1$ (resp. $K_{X_2} + D_2$) then $$H^0(X_1, L_1^{[m]}) = H^0(U, L_1^{[m]}) = H^0(U, L_2^{[m]}) =  H^0(X_2, L_2^{[m]}).$$
Therefore $X_1$ and $X_2$ are $\mathrm{Proj}$ of the same graded algebra, so they are all isomorphic.
\end{proof}

\section{Projectivity of the moduli of stable pairs}
\label{section projectivity ksba}
The goal of this section is to provide a different proof of the projectivity of the moduli of stable pairs, established in \cite{KP17}, using $\bQ$-pairs.
As a consequence, we also deduce the projectivity of the coarse moduli space of $\sF_{n,p,I}$.

We first construct a suitable polarization on the base of
families of $\bQ$-stable pairs with additional technical assumptions, see Theorem \ref{thm:main:step:proj}.
Then, we use this result to deduce the projectivity of the moduli of stable pairs, see Corollary \ref{cor proj moduli}.
Lastly, we deduce the projectivity of the coarse moduli space of $\sF_{n,p,I}$ from Corollary \ref{cor proj moduli} and Lemma \ref{lemma_bound_for_epsilon}, see Corollary \ref{cor_projectivity_our_moduli}.
In particular, one could alternatively deduce Corollary \ref{cor_projectivity_our_moduli} from Lemma \ref{lemma_bound_for_epsilon} and \cite{KP17}.

\begin{Lemma}\label{Lemma_there_is_a_section_vanishing_on_the_divisor}
With the notation of Theorem \ref{thm_morphism_our_stack_to_kollars_one_for_reduced_bases}, we will denote by $p\colon (\cX;\cD)\to B$ the
$\bQ$-stable morphism, and by $q \colon (\cY,\frac{1}{r}\cD_\cY)\to B$ the resulting stable family of Theorem \ref{thm_morphism_our_stack_to_kollars_one_for_reduced_bases}.
Assume this $\bQ$-stable morphism is a
family of $\bQ$-stable pairs, and let $m_0$ be the smallest positive integer such that $m_0\frac{1}{r}\cD$ is Cartier.
Then, there is $k_0$ divisible by $m_0$ such that, for every $k=\ell k_0$ positive multiple of $k_0$, the sheaf  $\cL_{\cY} \coloneqq \cO_\cY(k(K_{\cY/B} + \frac{1}{r}\cD_\cY))$ satisfies the following properties:
\begin{enumerate}[(a)]
    \item $\cL_\cY$ is Cartier;
    \item $R^iq_*(\cL_\cY^{\otimes j}) = 0$ for every $i>0$ and $j>0$;
    \item $\cL_\cY$ gives an embedding $\cY\hookrightarrow \bP(q_*\cL_\cY)$;
    \item $\cO_\cX(k(K_{\cX/B} + \frac{1}{r}\cD_\cY) - \frac{m_0}{r}\cD)$ is relatively ample; and
    \item for every $b\in B$, there is a section $s\in H^0(\cY_b, \cL_{\cY,b})$ such that $V(s)$ has codimension 1 in each irreducible component of $\cY_b$ and the scheme-theoretic image of $m_0\cD\to \cY$ restricted to $\cY_b$ is contained in $V(s)$.
\end{enumerate}
\end{Lemma}

\begin{Remark}
Observe that the definition of $m_0\frac{1}{r}\cD$ is given in Notation \ref{Notation_mD_for_pstable_fam}.
In our case, we can still define $m_0\frac{1}{r}\cD$ even if $(\cX;\cD)\to B$ was a $\bQ$-stable morphism instead, since the base $B$ is reduced, and Lemma \ref{Lemma_there_is_a_section_vanishing_on_the_divisor} would go through verbatim.
Since we will use Lemma \ref{Lemma_there_is_a_section_vanishing_on_the_divisor} only in the case in which $(\cX;\cD)\to B$ is a
family of $\bQ$-stable pairs, for simplicity we stick with the family case.
\end{Remark}

\begin{proof}
By construction, $K_{\cY/B} + \frac{1}{r}\cD_\cY$ is relatively ample.
Thus, for a sufficiently divisible $k_0$, properties (a) to (d) are satisfied for $k=k_0$.

Then, we can achieve (a) to (d) since if $\cO_\cY(k_0(K_{\cY/B} + \frac{1}{r}\cD_\cY))$ satisfies any condition between (a) to (e), then for every $m>0$ also $\cO_\cY(mk_0(K_{\cY/B} + \frac{1}{r}\cD_\cY))$ satisfies the same condition.
We only need to check that up to choosing $k_0$ divisible enough, also (e) holds.

Let $\cG \coloneqq \cO_\cY(k_0(K_{\cY/B} + \frac{1}{r}\cD_\cY))$ for a $k_0$ which satisfies (a)-(d), and we need to show that
$\cG^{\otimes m}$ satisfies (a)-(e) for some $m\gg 0$.
First observe that from cohomology and base change and point (b), for every $b\in B$ we have $q_*(\cG^{\otimes m})\otimes k(b) = H^0(\cY_b,\cG_{|\cY_b}^{\otimes m})$.

Let $\cZ\subseteq \cY$ be the scheme-theoretic image of $\pi_{|\frac{m_0}{r}\cD}$.
Up to replacing $B$ with a locally closed stratification $B'\to B$, we can assume that $\cZ' \coloneqq \cZ\times_B B'\to B'$ is flat, and $q_*\cG$ is free.
Then, we define $\cY' \coloneqq \cY\times_B B'$, the first projection $\pi_1 \colon \cY'\to \cY$, the second projection $q' \colon \cY'\to B'$, and $\cG' \coloneqq \pi_1^*\cG$.
Then, consider the exact sequence 
\[
0\to \cI\to \cO_{\cY'}\to \cO_{\cZ'} \to 0.
\]
We twist the exact sequence above by $\cG'^{\otimes m}$ and we obtain 
\[
0\to \cI\otimes \cG'^{\otimes m} \to  \cG'^{\otimes m}\to \cO_{\cZ'}\otimes \cG'^{\otimes m} \to 0.
\]
Since $\cZ'\to B'$ is flat, also $\cI$ is flat over $B'$. Therefore since $\cG'$ is relatively ample, up to choosing $m$ big enough, the following is an exact sequence on $B'$:

\[
0\to q'_*(\cI\otimes \cG'^{\otimes m})\to  q'_*(\cG'^{\otimes m})\to q'_*(\cO_{\cZ'}\otimes \cG'^{\otimes m}) \to 0.
\]
By \cite{ACH11}*{Corollary 4.5 (iii)}, there is $n_0$ such that for every
$n \ge n_0$ and any $a \ge 1$,
the multiplication map
\[
\Sym^a(q'_*\cG')\otimes q'_*(\cI\otimes \cG'^{\otimes n}) \to q'_*(\cI\otimes \cG'^{\otimes n+a})
\]
is surjective.
Then for every $b\in B'$, also 
$$\Sym^a(q'_*\cG')\otimes q'_*(\cI\otimes \cG'^{\otimes n}) \otimes k(b) \cong( \Sym^a(q'_*\cG')\otimes k(b) )\otimes_{k(b)}( q'_*(\cI\otimes \cG'^{\otimes n}) \otimes k(b)) \to q'_*(\cI\otimes \cG'^{\otimes n+a})\otimes k(b)$$
is surjective.
From \cite{ACH11}*{Corollary 4.5 (iv)} (or cohomology and base change), for $n \geq n_0$ and $a,s \ge 1$,
we have:
\[
\Sym^a(q'_*\cG')\otimes k(b) \cong H^0(\cO_{\bP^N}(a)),
\]
\[ q'_*(\cI\otimes \cG'^{\otimes n}) \otimes k(b) \cong H^0(\cY_b,(\cI\otimes \cG'^{\otimes n})_{|\cY_b}) = H^0(\cY_b,\cI_{|\cY_b }(n)),\text{ } \text{ }\text{ and }
\]
\[
q'_*(\cI\otimes \cG'^{\otimes n+s}) \otimes k(b) \cong H^0(\cY_b,(\cI\otimes \cG'^{\otimes n+s})_{|\cY_b}) = H^0(\cY_b,\cI_{|\cY_b }(n+s)).
\]
In other terms, for every $b\in B$ the ideal $\cI_{\cY_b}$ is generated in degree $n_0$. Then for every $b\in B$ there is a section $s\in H^0(\cY_b, \cL_{\cY,b}) = H^0(\cY_b, \cG^{\otimes n_0}_{\cY,b})$ (and therefore also a section $s^{\otimes m}\in H^0(\cY_b, \cL_{\cY,b}^{\otimes m})$ if $m \ge n_0$) such that $V(s)$ does not contain any irreducible component of $\cY_b$ and the scheme-theoretic image of $\frac{m_0}{r}\cD\to \cY$ restricted to $\cY_b$ is contained in $V(s)$.

Therefore, for the locally closed subset $B'$, we may choose $k_0'=k_0 n_0$, where $k_0$ was chosen at the beginning of the proof to satisfy (a)-(d) and thus define $\mathcal{G}$.
Then, by noetherianity, there are finitely many locally closed subsets $B'$ in the decomposition of $B$, and we may thus choose $k_0''$ as the least common multiple of the $k_0'$ defined on each $B'$.
\end{proof}

\begin{Cor}\label{corollary_the_section_that_vanishes_on_the_scheme_theoretic_image_of_m0/rD_can_be_lifted_on_X}
With the notation and assumptions of Lemma \ref{Lemma_there_is_a_section_vanishing_on_the_divisor}, for every $b\in B$ and every $a>0$, there is a global section $t$ of $\cL_{|\cX_b}^{\otimes a}$ that is not zero on the generic points of $\cX_b$ but its maps to 0 via the restriction map $H^0(\cX_b, \cL_{|\cX_b}^{\otimes a})\to H^0(\frac{am_0}{r}\cD_b, \cL^{\otimes a}_{|\frac{am_0}{r}\cD_b})$.
\end{Cor}

\begin{proof}
From Lemma \ref{Lemma_there_is_a_section_vanishing_on_the_divisor}, for every $b \in B$, there is a section $s$ that does not vanish on the generic points of $\cY_b$ but vanishes on the restriction to $\cY_b$ of the scheme theoretic image of $\frac{m_0}{r}\cD$.
From point (b) of  Lemma \ref{Lemma_there_is_a_section_vanishing_on_the_divisor}, $q_*\cL_\cY$ is a vector bundle on $B$, so there is an open subset $b\in U\subseteq B$ such that $(q_*\cL_\cY)_{|U}$ is free.
Then, we can extend the section $s$ to a global section $s'\in H^0(U,q_*(\cL_\cY)_{|U}) = H^0(U,(q_U)_*(\cL_{\cY,U}))$.
But $\pi$ is a contraction, so $\pi_{U} \colon \cX_U\to \cY_U$ is still a contraction, and $\pi_{U}^*\cL_{\cY,U} = \cL_U$.
In particular, consider the section $t\in H^0(\cX_U,\cL_U)$ that is the pull-back of $s'$.
Then, $t$ does not vanish on the generic points of $\cX_b$, as by Lemma \ref{ample model p-pair} those are in bijection with the generic points of $\cY_b$.
Moreover, it vanishes along $\frac{m_0}{r}\cD_b$: this is a local computation. Indeed, if we replace $\cY$ and $\cX$ with appropriate open subsets $\spec(A)$, $\spec(B)$; and $f$ is the generator for the ideal sheaf of $\frac{m_0}{r}\cD$, we have the following commutative diagram:
$$\xymatrix{B\otimes k(b)/f\otimes 1 & B/f \ar[l] \\ B\otimes k(b) \ar[u] & t\in B  \ar[u] \ar[l] \\ A\otimes k(b) \ar[u] & s\in A. \ar[u] \ar[l]}$$
Since the image of $s$ vanishes in $B\otimes k(b)/f\otimes 1 $, so does the image of $t$. 

Therefore for every $a>0$, the section $t^{\otimes a}\in H^0(\cX_U,\cL_U^{\otimes a})$ will vanish along $\frac{am_0}{r}\cD_b$.
In other terms, if we look at the exact sequence
\[
0 \to H^0(\cX_b, \cL_{|\cX_b}^{\otimes a}\otimes \cI_{\cD_b}^{[\frac{am_0}{r}]})\to H^0(\cX_b, \cL_{|\cX_b}^{\otimes a})\xrightarrow{\psi} H^0(\frac{am_0}{r}\cD_b, \cL^{\otimes a}_{|\frac{am_0}{r}\cD_b}),
\]
we have an element of $H^0(\cX_b, \cL_{|\cX_b}^{\otimes a})$ (namely, the restriction of $t^{\otimes a}$ to $\cX_b$) that does not vanish along the generic points of $\cX_b$ but maps to 0 via $\psi$.
\end{proof}

\begin{Lemma}\label{lemma_lc_restricted_to_mD_pushes_forward_to_a_vector_bundle}
With the notation and assumptions of Lemma \ref{Lemma_there_is_a_section_vanishing_on_the_divisor}, we will denote by $p\colon (\cX;\cD)\to B$ the
family of $\bQ$-stable pairs, and by $q \colon (\cY,\frac{1}{r}\cD_\cY)\to B$ the resulting stable family of Theorem \ref{thm_morphism_our_stack_to_kollars_one_for_reduced_bases}.
Let $k$ be a multiple of $k_0$, where $k_0$ is as in Lemma \ref{Lemma_there_is_a_section_vanishing_on_the_divisor}.
Also, set $\cL \coloneqq \cO_\cX(k(K_{\cX/B} + \frac{1}{r}\cD))$.

Then, there is an $a_0$ such that, for every $a\ge a_0$, the sheaf $p_*(\cL_{|\frac{am_0}{r}\cD}^{\otimes a})$ is a vector bundle, and fits in an exact sequence as follows: 
\[
0\to p_*(\cO_\cX(ka(K_{\cX/B} + \frac{1}{r}\cD) - \frac{am_0}{r}\cD)) \to p_*( \cL^{\otimes a}) \to p_*(\cL_{|\frac{am_0}{r}\cD}^{\otimes a})\to 0.
\]
\end{Lemma}

\begin{proof}
First, from relative Serre's vanishing and Lemma \ref{Lemma_there_is_a_section_vanishing_on_the_divisor} (b) and (d), there is an integer $a_0$ such that for every $b\in B$ and $a \geq a_0$, we have
    \begin{itemize}
        \item $H^i(\cO_{\cX_b}(ak(K_{\cX_b} + \frac{1}{r}\cD_b) - \frac{am_0}{r}\cD_b)=0$ for $i>0$; and
        \item $H^i(\cO_{\cY_b}(ak(K_{\cY_b} + \frac{1}{r}\cD_{\cY,b})))=0$ for $i>0$.
    \end{itemize}

Fix a positive integer $a>a_0$.
We begin by considering the following exact sequence, where $\cJ$ is the locally free ideal sheaf $\cO_\cX\left(-\frac{m_0}{r}\cD\right) $:
$$0\to \cL^{\otimes a}\otimes \cJ^{\otimes a} \to \cL^{\otimes a} \to \cL_{|\frac{am_0}{r}\cD}^{\otimes a}\to 0.$$
We push it forward via $p$, and from the first bullet point and from cohomology and base change,
$R^1p_*(\cL^{\otimes a}\otimes \cJ^{\otimes a}) = 0$, so we have:
\begin{equation} \label{eq_ses}
    0\to p_*(\cL^{\otimes a}\otimes \cJ^{\otimes a} ) \to p_*( \cL^{\otimes a}) \to p_*(\cL_{|\frac{m_0a}{r}\cD}^{\otimes a})\to 0.
\end{equation}

Recall that there is a contraction $\pi \colon \cX\to \cY$ such that, if we denote $\cL_\cY \coloneqq \cO_{\cY}(k(K_{\cY/B} + \frac{1}{r}\cD_\cY))$, then  $\pi^*(\cL_\cY) = \cL$.
In particular
\begin{equation} \label{eq_push_fwd}
    p_*(\cL^{\otimes a}) = q_*\pi_* (\cL^{\otimes a}) = q_*(\cL_\cY^{\otimes a}).
\end{equation}

Now, pick $b\in B$. From the first bullet point and from cohomology and base-change, we have \begin{equation} \label{eq_13}
    p_*(\cL^{\otimes a}\otimes \cJ^{\otimes a} )\otimes k(b)\cong H^0(\cX_b, (\cL^{\otimes a}\otimes \cJ^{\otimes a})_{|\cX_b})  \cong H^0(\cX_b,\cO_{\cX_b}(ak(K_{\cX_b} + \frac{1}{r}\cD_b) - \frac{am_0}{r}\cD_b))
\end{equation}
and from the second bullet point, equation \eqref{eq_push_fwd} and the fact that $\pi_b\colon \cX_b\to \cY_b$ is a contraction,

\begin{equation} \label{eq_14}
    p_*(\cL^{\otimes a})\otimes k(b)\cong H^0(\cY_b, \cL^{\otimes a}_{\cY, b}) \cong  H^0(\cY_b,\cO_{\cY_b}(ak(K_{\cY_b} + \frac{1}{r}\cD_b)))  \cong H^0(\cX_b, \cL^{\otimes a}_{|\cX_b}).
\end{equation}

Now, we tensor the sequence \eqref{eq_ses} by $k(b)$, and we obtain 
\[
     p_*(\cL^{\otimes a}\otimes \cJ^{\otimes a})\otimes k(b) \to p_*( \cL^{\otimes a})\otimes k(b) \to p_*(\cL_{|\frac{m_0a}{r}\cD}^{\otimes a})\otimes k(b)\to 0.
\]
From the natural adjunction between pull-back and push-forward, we have the following commutative diagram:
\[
\xymatrix{ & p_*(\cL^{\otimes a}\otimes \cJ^{\otimes a})\otimes k(b) \ar[d]_{\cong \text{ from \eqref{eq_13}}}\ar[r]^-\phi & p_*( \cL^{\otimes a})\otimes k(b) \ar[r] \ar[d]^{\cong \text{ from \eqref{eq_14}}} &p_*(\cL_{|\frac{m_0a}{r}\cD}^{\otimes a})\otimes k(b)\ar[r] \ar[d] & 0
\\ 0 \ar[r] & H^0( \cX_b, (\cL^{\otimes a}\otimes \cJ^{\otimes a})_{|\cX_b})\ar[r] & H^0(\cX_b, \cL^{\otimes a}_{|\cX_b}) \ar[r] & H^0(\frac{m_0a}{r}\cD, (\cL^{\otimes a})_{|\frac{m_0a}{r}\cD_b}) \ar[r]^-{(\ast)}& 0}
\]
Observe that the exactness of the map $(\ast)$ follows from the first bullet point.
Then, from diagram chasing, also $\phi$ is injective.
In particular,
\[
\dim_{k(b)}(p_*(\cL_{|\frac{m_0a}{r}\cD}^{\otimes a})\otimes k(b)) = \dim_{k(b)}(H^0(\cX_b, \cL^{\otimes a}_{|\cX_b})) - \dim_{k(b)}(H^0( \cX_b, (\cL^{\otimes a}\otimes \cJ^{\otimes a})_{|\cX_b})).
\]
However the right-hand side does not depend on $b\in B$. Therefore also the left-hand side does not depend on $b\in B$, so $p_*(\cL_{|\frac{m_0a}{r}\cD}^{\otimes a})$ is a vector bundle.
\end{proof}

\begin{Cor}\label{cor_determinant}
With the notation and assumptions of Lemma \ref{lemma_lc_restricted_to_mD_pushes_forward_to_a_vector_bundle}, there is an isomorphism
\[
\det p_*(\cL^{\otimes a}) \cong \det (p_*(\cO_{\cX}(ak(K_{\cX/B} + \frac{1}{r}\cD) - \frac{am_0}{r}\cD)))\otimes \det( p_*(\cL_{|\frac{m_0a}{r}\cD}^{\otimes a})).
\]
\end{Cor}

\begin{theorem}\label{thm:main:step:proj}
With the notation and assumptions of Lemma \ref{lemma_lc_restricted_to_mD_pushes_forward_to_a_vector_bundle}, further assume that the map $B\to \sF_{n,p,I}$ given by the
family of $\bQ$-stable pairs $(\cX;\cD)\to B$ is finite.
Then, for $a$ divisible enough, the line bundle $\det p_*(\cL^{\otimes a})$ is ample.
\end{theorem}

\begin{proof}
We plan on using Koll\'ar's Ampleness Lemma, see \cite{Kol90}.
By Corollary \ref{cor_determinant}, it suffices to show that $\det p_*(\cL^{\otimes a}) \otimes \det (p_*(\cO_\cX(ak(K_{\cX/B} + \frac{1}{r}(1-\frac{m_0}{k})\cD))))\otimes \det( p_*(\cL_{|\frac{m_0a}{r}\cD}^{\otimes a}))$ is ample.

Consider the vector bundles
\[
\cQ_a \coloneqq p_*(\cL^{\otimes a}) \oplus p_*(\cO_\cX(ak(K_{\cX/B} + \frac{1}{r}(1-\frac{m_0}{k})\cD)))) \oplus p_*(\cL_{|\frac{m_0a}{r}\cD}^{\otimes a})\text{ }\text{ }\text{ }\text{ and }
\]
\[
\cW_{b,c} \coloneqq \Sym^b( p_*(\cL^{\otimes c})) \oplus \Sym^b(\cO_\cX(ck(K_{\cX/B} + \frac{1}{r}(1-\frac{m_0}{k})\cD)) \oplus \Sym^b( p_*(\cL^{\otimes c}))
\]
for appropriate choices of $a$, $b$, and $c$.

When $a = bc$, there is a morphism $\cW_{b,c}\to \cQ_{bc}$ given by the sum of the multiplication morphisms
\[
\Sym^b( p_*(\cL^{\otimes c})) \xrightarrow{\alpha} p_*(\cL^{\otimes bc}),\quad \Sym^b( p_*(\cO_\cX(ck(K_{\cX/B} + \frac{1}{r}(1-\frac{m_0}{k})\cD))) \xrightarrow{\beta} p_*(\cO_\cX(bck(K_{\cX/B} + \frac{1}{r}(1-\frac{m_0}{k})\cD)),
\]
and the composition of $\alpha$ with the surjection in the exact sequence of Lemma \ref{lemma_lc_restricted_to_mD_pushes_forward_to_a_vector_bundle}, denoted by
\[
\Sym^b( p_*(\cL^{\otimes c})) \xrightarrow{\gamma} p_*(\cL_{|\frac{m_0bc}{r}\cD}^{\otimes bc}).
\]

First, observe that for $c$ divisible enough, the vector bundles $p_*(\cL^{\otimes c})$ and $  p_*(\cO_\cX(ck(K_{\cX/B} + \frac{1}{r}(1-\frac{m_0}{k})\cD)))$ are nef. Indeed, from the assumptions of Lemma \ref{lemma_lc_restricted_to_mD_pushes_forward_to_a_vector_bundle}, the formation of $p_*(\cL^{\otimes c})$ and $  p_*(\cO_\cX(ck(K_{\cX/B} + \frac{1}{r}(1-\frac{m_0}{k})\cD)))$ commutes with base change, so we can assume that the base $B$ is a smooth curve.
Then the statement follows from \cite{Fuj18}*{Theorem 1.11}.
Therefore, their symmetric powers and sum are nef, so $\cW_{b,c}$ is nef for $c$ divisible enough.

Since $\cL_{\cY}$ is relatively ample and $q_*(\cL_{\cY}^{\otimes m}) \cong p_*(\cL^{\otimes m})$, there is a $b_0$ such that, for every $b\ge b_0$, the map $\alpha$ is surjective.
For the same reason, if we choose $c$ large enough so that $ck(K_{\cX/B} + \frac{1}{r}(1-\frac{m_0}{k})\cD)$ is relatively very ample, also $\beta$ is surjective for $b_0=b_0(c)$ sufficiently large (see  \cite{ACH11}*{Corollary 4.5}).
Then, also $\gamma$ is surjective since it is the composition of surjective morphisms.
Therefore we have a surjection $\Phi \colon \cW_{b,c}\to \cQ_{bc}$.
We denote by $G$ the structure group of $\cW_{b,c}$.
This gives a map of sets
\[
\Psi \colon |B|\to |[\Gr(w,q)/G]|,
\]
where $w$ (resp. $q$) is the rank of $\cW_{b,c}$ (resp. $\cQ_{bc}$), and where for a stack $\cZ$ we denote by $|\cZ|$ its associated topological space.
If we show that $\Psi$ has finite fibers, then the theorem follows from \cite{Kol90}*{Ampleness Lemma, 3.9}.

Consider two points $x_1$, $x_2$ that map to the same point via $\Psi$.
Over the point $x_1$ we have the surjection $\Phi_{x_1} \colon \cW_{b,c}\otimes k(x_1)\to \cQ_{bc}\otimes k(x_1)$, and similarly we have one denoted by $\Phi_{x_2}$ over $x_2$.
Choose two isomorphisms
\[
\tau_1 \colon H^0(\bP^{N}, \cO_{\bP^{N}}(1)) \to p_*(\cL^{\otimes c})\otimes k(x_1), \text{ and }
\]
\[
\tau_1' \colon H^0(\bP^{M}, \cO_{\bP^{M}}(1)) \to  p_*(\cO_\cX(ck(K_{\cX/B} + \frac{1}{r}(1-\frac{M}{k})\cD)))  \otimes k(x_1).
\]
Similarly we define  $\tau_2$ and $\tau_2'$ for the same isomorphisms over $x_2$.
This gives an isomorphism
\[
H^0(\bP^{N}, \cO_{\bP^{N}}(b)) \oplus H^0(\bP^{M}, \cO_{\bP^{M}}(b)) \oplus H^0(\bP^{N}, \cO_{\bP^{N}}(b)) \xrightarrow{\Sym^b(\tau_1) \oplus \Sym^b(\tau_1') \oplus \Sym^b(\tau_1)} \cW_{b,c}\otimes k(x_1).
\]
Since $x_1$ and $x_2$ map to the same point via $\Psi$, using the identifications above, there is an element $g\in G$ such that $g\Ker( \Phi_{x_1}) =\Ker(\Phi_{x_2}) $.

In particular, we can choose a basis for $\cW_{b,c}\otimes k(x_1)$ by first choosing a basis for each summand of the left-hand side, choosing the same basis for the first and third summand.
Then, since $\cW_{b,c}$ is a direct sum of vector bundles, in this basis the linear transformation $g$ is block diagonal:
\[
g = \begin{bmatrix}A & 0 & 0 \\ 0 & B & 0 \\ 0 & 0 & A\end{bmatrix}.
\]
In particular, $g$ will send $\Ker(\alpha_{x_1})$ (resp. $\Ker(\gamma_{x_1})$) to $\Ker(\alpha_{x_2})$ (resp. $\Ker(\gamma_{x_2})$).
But the kernel of $\alpha_{x_1}$ corresponds to the symmetric functions of degree $b$ that vanish on a subvariety of $\bP^{n_1}$ isomorphic to $\cY_{x_1}$, which we will still denote by $\cY_{x_1}$.
From \cite{ACH11}*{Corollary 4.5}, up to choosing $b_0$ big enough, this kernel generates a graded ideal that corresponds to a subvariety isomorphic to $\cY_{x_1}$.
The same conclusion holds for $\Ker(\alpha_{x_2})$, since there is a linear transformation (given by a block of the matrix $A$) that induces an isomorphism $\Ker(\alpha_{x_1})\to \Ker(\alpha_{x_2})$.
So, in particular, this linear transformation induces a map of projective spaces, that gives an isomorphism $h \colon \cY_{x_1}\to \cY_{x_2}$.

By Corollary \ref{corollary_the_section_that_vanishes_on_the_scheme_theoretic_image_of_m0/rD_can_be_lifted_on_X}, $\Ker(\gamma_{x_1})$ contains a function that does not vanish along the generic points of the irreducible components of $\cY_{x_1}$ but, if pulled back via $\pi_{x_1} \colon \cX_{x_1}\to \cY_{x_1}$, it vanishes along $\frac{m_0a}{r}\cD_{x_1}$.
Therefore, the zero locus of the polynomials in $\Ker(\gamma_{x_1})$ has codimension 1 in $\cY_{x_1}$ and the union of its irreducible components of codimension 1, which we denote by $\Gamma_{x_1}$, contains $\supp(\frac{1}{r}\cD_{\cY,x_1})$.
Since $\Gamma_{x_1}$ has finitely many irreducible components of codimension one and the coefficients of $\cD_{\cY,x_1}$ are in the finite set $rI$, the divisor $\cD_{\cY,x_1}$ is determined up to finitely many possible choices of prime divisors and coefficients. 

A similar conclusion holds by replacing $x_1$ with $x_2$. Since in the description of $g$ the block at position (1,1) is the same as the block at position (3,3), we also know that $h(\Gamma_{x_1}) = \Gamma_{x_2}$.

Therefore the fiber of $\Psi(x_1)$ corresponds to stable pairs $(Y,D)$ in our moduli problem where $Y\cong \cY_{x_1}$, $\Supp(D)\subseteq \Supp(\Gamma_{x_1})$.
But there are finitely many such subvarieties, and since $B\to \sF_{n,p,I}$ is finite by our assumptions, the fiber $\Psi(x_1)$ has to be finite.
\end{proof}

\begin{Cor} \label{cor proj moduli}
Consider a proper DM stack $ \cK_{n,v,I}$ which satisfies the following two conditions:
\begin{enumerate}
    \item for every normal scheme $S$, the data of a morphism $f \colon S\to \cK_{n,v,I}$ is equivalent to a stable family of pairs $q \colon (\cY,\cD)\to B$ with fibers of dimension $n$, volume $v$ and coefficients in $I$; and
    \item there is $m_0\in \mathbb{N}$ such that, for every $k \in \mathbb{N}$, there is a line bundle $\cL_k$ on $\cK_{n,v,I}$ such that, for every morphism $f$ as above, $f^*\cL_k \cong \det(q_*(\omega_{\cY/B}^{[km_0]}(km_0\cD)))$.
\end{enumerate}
Then, for $m_0$ divisible enough, $\cL_k$ descends to an ample line bundle on the coarse moduli space of $\cK_{n,v,I}$ for every $k\ge 1$. In particular, the coarse moduli space of $\cK_{n,v,I}$ is projective. 
\end{Cor}

\begin{Remark}
As a concrete example of Corollary \ref{cor proj moduli}, one can consider $\cK_{n,v,I}$ to be any moduli space of KSB-stable pairs $\sK$, such that, if we denote by $\pi \colon (\sY,\sD)\to \sK$ its universal family, then $\omega_{\sY/\sK}^{[m_0]}(m_0\sD)$ is $\pi$-ample for $m_0$ divisible enough.
Indeed, in this case, by cohomology and base change $\pi_*\omega_{\sY/\sK}^{[m_0]}(m_0\sD)$ is a vector bundle for $m_0$ divisible enough, its formation commutes with base change from cohomology and base change, and the formation of the determinant commutes with base change as well.
\end{Remark}

\begin{proof}
The argument is divided into two steps. We denote by $K_{n,v,I}$ the coarse moduli space of $\cK_{n,v,I}$, and let $\cM$ be an irreducible component of it.

\begin{bf}Step 1.\end{bf} There is a $\bQ$-stable morphism $(\cX;\cD)\to B$ which satisfies the following conditions:
\begin{enumerate}
    \item $B$ is normal and projective,
    \item The map $B\to \sF_{n,p,I}$ is finite,
    \item There is a dense open subset $U\subseteq B$ and a stable family $(\cY_U,\cD_{\cY, U})$ which satisfies the assumptions of Theorem \ref{thm_morphism_our_stack_to_kollars_one_for_reduced_bases}, and
    \item The family $q \colon (\cY_U,\cD_{\cY, U})\to U$ induces a map $U\to \cK_{n,v,I} \to K_{n,v,I}$ which dominates $X$.
\end{enumerate}

Consider the generic point of $\cM$, and consider the stable pair $(\cZ_\eta, \cD_{\cZ,\eta})$ it corresponds to. By Lemma \ref{lift stable pair slc}, there is the spectrum of a field $\spec(\mathbb{F})$ and a $\bQ$-stable family over it $f \colon (\cX_\eta; \cD_\eta) \to \spec(\mathbb{F})$ such that the
canonical model of $(\cX_\eta; \cD_\eta)$ is $(\cZ_\eta, \cD_{\cZ,\eta})$. The family $f$ induces a morphism $\spec(\mathbb{F}) \to \sF_{n,p,I}$, and we can take $\cF$ to be the closure of its image. This is still a proper DM stack, it admits a finite and surjective cover by a scheme $B'\to \cF$ with $B'$ a normal and \emph{projective} scheme using \cite{LMB18}*{Theorem 16.6}, Chow's lemma and potentially normalizing. Up to replacing $\mathbb{F}$ with a finite cover of it, we can lift $\spec(\mathbb{F})\to \cF$ to $\spec(\mathbb{F})\to B'$, and consider $B$ an irreducible component of $B$ containing the image of $\spec(\mathbb{F})\to B'$.

The morphism $B\to \sF_{n,p,I}$ induces a $\bQ$-stable family $(\cX;\cD)\to B$, and by construction its generic fiber admits an canonical model.
Then such an canonical model can be spread out: there is an open subset $U\subseteq B$ and a family $(\cY_U,\cD_{\cY,U})$ as in Theorem \ref{thm_morphism_our_stack_to_kollars_one_for_reduced_bases}, and the generic fiber $(\cY_U,\cD_{\cY,U})\to U$ is isomorphic to $(\cZ_\eta, \cD_{\cZ,\eta})$. Therefore the image of the corresponding map $U\to \cK_{n,v,I}\to K_{n,v,I}$ contains the generic point of $\cM$, and since both $U$ and $\cM$ are irreducible, $U\to X$ is dominant.

\begin{bf}Step 2.\end{bf} The map $U\to \cK_{n,v,I}$ extends to a finite map $\Phi \colon B\to \cK_{n,v,I}$.

The extension follows from Theorem \ref{thm_morphism_our_stack_to_kollars_one_for_reduced_bases}. To prove that $\Phi$ is finite, we can use Lemma \ref{lemma:finitedness:paper:zsolt}. Indeed, if it was not finite, there would be a curve contained in a fiber of $\Phi$. But then, up to replacing $C$ with an open subset of it, there would be a stable pair $(Y,D_Y)$ and a $\bQ$-stable family $g \colon (\cX;\cD)\to C$ as in Lemma \ref{lemma:finitedness:paper:zsolt}. Therefore there would be finitely many isomorphism types of $\bQ$-stable pairs in the fibers of $g$: this contradicts point (2) above.

\begin{bf}End of the proof.\end{bf} From Theorem \ref{thm:main:step:proj}, up to replacing $m_0$ with a multiple (which depends only on $\cM$), the line bundle $\Phi^*\cL_k$ is ample. But a multiple of $\cL_k$ descends to a line bundle on $K_{n,v,I}$. In other terms, there is a line bundle $G$ on $K_{n,v,I}$ whose pull-back is $\cL_k^{\otimes c}$ for a certain $c>0$. Therefore $G|_{\cM}$ is ample since it is ample once pulled-back via the \emph{finite} map $B\to \cM$. But if a line bundle is ample when restricted to each irreducible component, it is ample.
\end{proof}

\begin{Cor}\label{cor_projectivity_our_moduli} The stack $\sF_{n,p,I}$ admits a projective coarse moduli space.
\end{Cor}
\begin{proof}
We will denote by $F_{n,p,I}$ the coarse moduli space of $\sF_{n,p,I}$, which exists from Keel--Mori's theorem, and with $F_{n,p,I}^n$ the one of the normalization $\sF_{n,p,I}^n$ of $\sF_{n,p,I}$.

Consider $\epsilon_0$ such that $K_X+(1-\epsilon_0)D$ is ample for every $\bQ$-pair $(X;D)$ parametrized by $\sF_{n,p,I}$. Recall that such an
$\epsilon_0$ exists from Lemma \ref{lemma_bound_for_epsilon} (or from boundedness). In particular, if we denote with $(\sX;\sD)\to \sF_{n,p,I}$ the universal family, for $m$ divisibie enough the formation of $\sG \coloneqq  \det(p_*(\cO_{\sX}(m(K_{\sX/\sF_{n,p,I}} + \frac{1-\epsilon}{r}\sD))))$ commutes with base change and an high enough power of $\sG$ descends to a line bundle $G$ on $F_{n,p,I}$. Therefore to prove that $G$ is ample, it suffices to replace $F_{n,p,I}$ (resp. $G$) with $F_{n,p,I}^n$ (resp. the pull-back $G^n$ of $G$ via $F_{n,p,I}^n \to F_{n,p,I}$).

Consider a proper DM stack $\cK_{n,p(1-\epsilon_0),I}$ as in Corollary \ref{cor proj moduli}. When $D$ has coefficient $1-\epsilon_0$ the pairs parametrized by $\sF_{n,p,I}^n$ are stable of volume $p(1-\epsilon_0)$, so we have a map $\sF_{n,p,I}^n\to \cK_{n,p(1-\epsilon_0),I}$ which is finite, as different points of $\sF_{n,p,I}$ parametrize different $\bQ$-pairs, and the normalization is a finite morphism. From Corollary \ref{cor proj moduli}, the formation of $\cL_k$ commutes with base change, so $\cL_k$ pulls back to $G^n$. Then $G^n$ is ample as it is the pull-back of an ample line bundle via a finite morphism.\end{proof}

\bibliographystyle{amsalpha}
\bibliography{jag}

\end{document}